\documentclass[10pt]{amsart}

\usepackage{textcomp}
\usepackage{amssymb}
\usepackage{tikz-cd}
\usepackage[all]{xy}
\usepackage{enumerate}

\newcommand{\rar}[1]{\stackrel{#1}{\longrightarrow}}

\newcommand{\hush}{\natural}
\newcommand{\idot}{{\:\raisebox{1pt}{\text{\circle*{1.5}}}}}

\newcommand{\At}{{A_{\idot}}}
\newcommand{\Bt}{{B_{\idot}}}
\newcommand{\Dt}{{D_{\idot}}}
\newcommand{\Mt}{{M_{\idot}}}
\newcommand{\Vt}{{V_{\idot}}}
\newcommand{\Pt}{{P_{\idot}}}

\newcommand{\vH}{\check{H}}

\newcommand{\bC}{{\mathbb C}}

\newcommand{\bF}{{\mathbb F}}

\newcommand{\bZ}{{\mathbb Z}}

\newcommand{\cA}{{\mathcal A}}
\newcommand{\cB}{{\mathcal B}}
\newcommand{\cC}{{\mathcal C}}

\newcommand{\cE}{{\mathcal E}}
\newcommand{\cF}{{\mathcal F}}

\newcommand{\cL}{{\mathcal L}}
\newcommand{\cM}{{\mathcal M}}

\newcommand{\cO}{{\mathcal O}}

\newcommand{\cV}{{\mathcal V}}

\newcommand{\cd}{{\delta}}

\newcommand{\ks}{{\tilde{\kappa}}} 
\newcommand{\CPt}{{CP_{\idot}}}
\newcommand{\Ct}{{CH_{\idot}}}

\newcommand{\Sb}{\overline{S}}

\newcommand{\tA}{\tilde{A}}

\newcommand{\td}{\tilde{d}}
\newcommand{\ti}{\tilde{i}}

\newcommand{\tn}{\tilde{n}}

\newcommand{\tR}{\tilde{R}}
\newcommand{\tF}{\tilde{F}}
\newcommand{\tX}{\tilde{X}}

\newcommand{\tx}{\widetilde{x}}

\def\tn{\widetilde{\nabla}}
\def\td{\widetilde{d}}

\def\tx{\widetilde{x}}

\newcommand{\nc}{\newcommand}

\nc\wh{\widehat}

\nc\on{\operatorname}

\nc\Gr{\on{Gr}}

\nc\Fl{\on{Fl}}

\DeclareMathOperator{\Id}{{Id}}

\newcommand{\limto}{{\displaystyle\lim_{\longrightarrow}}}
\newcommand{\rightlim}{\mathop{\limto}}

\newcommand{\leftlim}{\mathop{\displaystyle\lim_{\longleftarrow}}}
\newcommand{\limfromn}{\leftlim\limits_{\raise3pt\hbox{$n$}}}
\newcommand{\limton}{\rightlim\limits_{\raise3pt\hbox{$n$}}}

\newcommand{\rightlimit}[1]{\mathop{\lim\limits_{\longrightarrow}}\limits%
                    _{\raise3pt\hbox{$\scriptstyle #1$}}}

\newcommand{\leftlimit}[1]{\mathop{\lim\limits_{\longleftarrow}}\limits%
                    _{\raise3pt\hbox{$\scriptstyle #1$}}}

\newcommand{\iso}{\buildrel{\sim}\over{\longrightarrow}}

\newcommand{\mono}{\hookrightarrow}

\DeclareMathOperator{\Spec}{{Spec}}

\DeclareMathOperator{\Res}{{Res}}

\makeatletter

\newcommand{\Rmnum}[1]{\expandafter\@slowromancap\romannumeral #1@}
\makeatother

\newtheorem*{theo}{Theorem}

\newtheorem{Th}{Theorem}
\newtheorem{pr}{Proposition}[section]
\newtheorem{lm}[pr]{Lemma}

\newtheorem{cor}[pr]{Corollary}

\theoremstyle{definition}

\newtheorem{rem}[pr]{Remark}
\newtheorem{remarks}[pr]{Remarks}

\numberwithin{equation}{section}

\begin{document}

\title[The Gauss-Manin connection on the periodic cyclic homology]
{The Gauss-Manin connection on the periodic cyclic homology}
\dedicatory{To Sasha Beilinson on his 60th birthday, with admiration}
%\author{Alexander Petrov \quad Dmitry Vaintrob \quad Vadim Vologodsky}

%\address{Department of Mathematics, University of Oregon, Eugene, OR, 97403, USA}
%\email{allens@uoregon.edu, vvologod@uoregon.edu}

%and
%\affilnum{2}Complete Second Author Address}

% Address / e-mail address of corresponding author
%\correspdetails{corr.email@math.edu}
\author[A.~Petrov]{Alexander Petrov}
\address{National Research University ``Higher School of Economics'',  Russian Federation
}
\email{aapetrov3@yandex.ru}
\author[ D. ~Vaintrob]{ Dmitry Vaintrob}
\address{Institute for Advanced Study, USA
}
\email{mvaintrob@gmail.com}
\author[ V. ~Vologodsky]{ Vadim Vologodsky}
\address{National Research University ``Higher School of Economics'',  Laboratory of Mirror Symmetry NRUHSE,  Russian Federation,  and  University of
Oregon, USA}
\email{vvologod@uoregon.edu}

%\keywords{Deformations of algebras, Poisson algebras, Witt vectors.}

%\subjclass[2010]{Primary 14F10, 14G17;  Secondary 16S34, 16S80.}

%\date{}
\begin{abstract} 
Let $R$ be the algebra of functions on a smooth affine irreducible curve $S$ over a field $k$ and let $\At$ be smooth and proper DG algebra over $R$. The relative periodic cyclic homology $HP_* (\At)$ of $\At$ over $R$ is equipped with the Hodge filtration $\cF^{\cdot}$ and the Gauss-Manin connection $\nabla$
(\cite{ge}, \cite{ka1}) satisfying the Griffiths transversality condition. When $k$ is a perfect field of odd characteristic $p$, we prove that, if the relative Hochschild homology  $HH_m(\At, \At)$  vanishes in degrees $|m| \geq p-2$,
then  a lifting $\tR$ of $R$  over $W_2(k)$ and a lifting of  $\At$ over $\tR$ determine the structure of a relative  Fontaine-Laffaille module (\cite{fa}, \S 2 (c),  \cite{ov} \S 4.6) on $HP_* (\At)$. That is, the inverse Cartier transform of the Higgs $R$-module $(Gr^\cF HP_* (\At), Gr^\cF \nabla)$ is canonically isomorphic to $ (HP_* (\At), \nabla)$. This is non-commutative counterpart of Faltings' result (\cite{fa}, Th. 6.2) for the de Rham cohomology of a smooth proper scheme over $R$.
 Our result amplifies the non-commutative Deligne-Illusie decomposition proven by Kaledin (\cite{ka3}, Th. 5.1). As a corollary, we show that the $p$-curvature of the Gauss-Manin connection on $HP_* (\At )$ is nilpotent and, moreover, it can be expressed in terms of the Kodaira-Spencer class $\kappa \in HH^2(A, A) \otimes _R \Omega^1_R$ (a similar result for the $p$-curvature of the Gauss-Manin connection on the de Rham cohomology is proven by Katz in \cite{katz2}).  As an application of the nilpotency of the $p$-curvature we prove, using a result from \cite{katz1}),  a version of ``the local monodromy theorem'' of  Griffiths-Landman-Grothendieck for the periodic cyclic homology: if $k=\bC$,
$\Sb$ is a smooth compactification of $S$,
then, for any smooth and proper  DG algebra  $\At$ over $R$, the Gauss-Manin connection on the relative periodic cyclic homology $HP_* (\At)$ has regular  singularities, and its monodromy around every point at  $\Sb -S$ is quasi-unipotent.

%Let  $\Sb$ be a smooth curve over $\bC$, $s_0\in \Sb$ a closed point,  and let  $\pi: X\to S$ be a  smooth proper scheme over $S=\Sb -s_0$.  The Griffiths-Landman-Grothendieck ``Local Monodromy Theorem'' asserts that
%the Gauss-Manin connection on the relative de Rham cohomology $H^*_{DR}(X/S)$ has a regular  singularity at $s_0$ and that its local monodromy  around $s_0$ is quasi-unipotent.  We propose a conjectural generalization of this result,
%where the de Rham cohomology is replaced by the periodic cyclic homology of a smooth proper DG algebra over $S$ equipped with the Gauss-Manin-Getzler connection. We discuss a possible approach to a proof of this conjecture
%based on the reduction to characteristic $p$ technique and ideas from  (\cite{katz}) and (\cite{kaledin}).

 \end{abstract}

\maketitle

\section{Introduction}
It is expected that the periodic cyclic homology of a DG algebra over $\bC$ (and, more generally,  the periodic cyclic homology of a DG category) carries a lot of additional structure similar to the mixed Hodge structure on the de Rham cohomology of algebraic varieties. Whereas a construction of such a structure
seems to be out of reach at the moment its counterpart in finite characteristic is much better understood thanks to recent groundbreaking works of Kaledin. In particular, it is proven in \cite{ka3} that under some assumptions on a DG algebra $\At$ over a perfect field $k$ of characteristic $p$, a lifting of $\At$ over the ring of second Witt vectors $W_2(k)$ specifies the structure of a Fontaine-Laffaille module on the periodic cyclic homology of $\At$. The purpose of this paper is to develop a relative version of Kaledin's theory for DG algebras over a base $k$-algebra $R$ incorporating in the picture the Gauss-Manin connection on the  relative periodic cyclic homology constructed by Getzler in \cite{ge}.  Our main result, Theorem \ref{fmth}, asserts that, under some assumptions on $\At$,  the Gauss-Manin connection on its periodic cyclic homology can be recovered from the Hochschild homology of $\At$ equipped with the action of the Kodaira-Spencer operator as the inverse Cartier transform (\cite{ov}).
As an application, we prove, using the reduction modulo $p$ technique, that, for a smooth and proper DG algebra over a complex punctured disk,
the monodromy of  the Gauss-Manin connection
on its periodic cyclic homology is quasi-unipotent.

\subsection{Relative Fontaine-Laffaille modules.}\label{s.s.r.f.l.m.} Let $R$ be a finitely generated commutative  algebra over a perfect field $k$ of {\it odd} characteristic $p>2$. 
  Assume that $R$ is smooth over $k$. Recall from (\cite{fa}, \S 2 (c),  \cite{ov} \S 4.6)  the notion  of {\it relative Fontaine-Laffaille module}
\footnote{In \cite{fa}, Faltings does not give a name to these objects. In \cite{ov}, they are called Fontaine modules. The name suggested here is a tribute to 
\cite{fl}, where these objects were first introduced in the case when $R=k$.} over $R$.
Fix a flat lifting  $\tR$  of $R$  over the ring $W_2(k)$ of second Witt vectors and a lifting $\tF: \tR  \to \tR$ of the Frobenius morphism $F: R\to R$. Define the inverse Cartier transform 
$$\cC_{(\tR, \tF)}^{-1}: {\sf HIG}(R) \to {\sf MIC}(R)$$
to be a functor from the category of Higgs modules {\it i.e.}, pairs $(E, \theta)$, where $E$ is an $R$-module and 
$\theta:  E\to E\otimes _{R} \Omega^1_{R}$ is an $R$-linear morphism such that the composition $\theta ^2: E\to E\otimes _{R} \Omega^1_{R}\to E\otimes _{R} \Omega^2_{R}$ equals $0$\footnote{Equivalently, a Higgs module is a module over the symmetric algebra $S^{\idot}T_R$.},  to the category of $R$-modules with integrable connection.  Given a Higgs module $(E, \theta)$ we set
$$\cC_{(\tR, \tF)}^{-1}(E, \theta):= (F^*E,  \nabla_{can} + C^{-1}_{(\tR, \tF)}(\theta)),$$
where $\nabla_{can}$ is the Frobenius pullback connection on $F^*E$ and  the map
\begin{equation}\label{inversecartieroperator}
C^{-1}_{(\tR, \tF)}:  End_{R}(E)\otimes \Omega^1_{R} \to F_*(End_{R}(F^*E) \otimes _R \Omega^1_R)
\end{equation}
takes $f \otimes  \eta$ to $F^*(f) \otimes \frac{1}{p}\tF^* \tilde{\eta}$, with $\tilde{\eta}\in  \Omega^1_{\tR}$ being a lifting of $\eta$.  A relative Fontaine-Laffaille module
 over $R$ consists of a finitely generated $R$-module $M$ with an integrable connection $\nabla$ and a Hodge filtration
$$0=\cF^{l+1}M \subset \cF^{l}M \subset \cdots \subset  \cF^{m}M=M$$
satisfying the Griffiths transversality condition, together with isomorphism in ${\sf MIC}(R)$:
$$\phi: \cC_{(\tR, \tF)}^{-1} (\Gr_\cF^{\idot} M, \Gr_\cF {\nabla}) \iso (M, \nabla).$$
 Here $$\Gr_\cF {\nabla}: \Gr_\cF ^{\idot} M \to Gr_\cF ^{\idot -1} M $$ is the ``Kodaira-Spencer'' Higgs field induced by $\nabla$. \footnote{In \cite{fa}, Faltings considers more general objects. In fact, what we call
 here a relative Fontaine-Laffaille module is the same as a $p$-torsion object in Faltings' category $\cM\cF_{[m,l]}^\nabla(R)$}
 
 The category $\cM\cF_{[m,l]}(\tR, \tF)$ (where $l\geq m$ are arbitrary integers) of relative Fontaine-Laffaille modules has a number of remarkable properties not obvious from the definition. 
  It is proven by Faltings in (\cite{fa}, Th. 2.1) that $\cM\cF_{[m,l]}(\tR, \tF)$  is abelian, 
   every morphism between Fontaine-Laffaille modules is  strictly compatible with the Hodge filtration, and, for every Fontaine-Laffaille module $(M, \nabla, \cF^{\idot}M, \phi)$, the $R$-modules $M$ and $Gr_\cF M$ are flat. Moreover, if $l-m<p$, the category $\cM\cF_{[m,l]}(\tR, \tF)=: \cM\cF_{[m,l]}(\tR)$ is independent of the choice of the Frobenius lifting\footnote{Every two liftings $\tR$, $\overline R$ of $R$ are isomorphic. A choice of such an isomorphism induces an equivalence $\cM\cF_{[m,l]}(\tR)\iso \cM\cF_{[m,l]}(\overline R)$.
 We refer the reader to  (\cite{ov} \S 4.6) for a construction of the category of Fontaine-Laffaille modules over any smooth scheme $X$ over $k$ equipped with a lifting $\tilde X$ over $W_2(k)$.}.  Fontaine-Laffaille modules  arise geometrically: it is shown in (\cite{fa}, Th. 6.2) that, for a smooth proper scheme
 $X\to spec\, R$ of relative dimension less than $p$, a lifting of $X$ over $\tR$ specifies a Fontaine-Laffaille module structure on the relative de Rham cohomology $H_{DR}^{\idot}(X/R)$.
 
\subsection{The Kodaira-Spencer class of a DG algebra.} Let $\At$ be  a differential graded algebra  over $R$.  
Denote by $HH^{\idot}(\At, \At)$ its Hochschild cohomology and by
\begin{equation}\label{ksclassintro}
\kappa \in HH^2(\At, \At)\otimes _R \Omega^1_R
\end{equation}
 the {\it Kodaira-Spencer class}
 of $\At$. This can be defined as follows. Choose a quasi-isomorphism
$\At \iso \Bt $,
where $\Bt$ is a semi-free DG algebra over $R$ (\cite{dr}, \S 13.4) and
a connection $\nabla ':  \bigoplus B_i \to \bigoplus B_i \otimes \Omega^1_R$ on the graded algebra $\bigoplus B_i$ satisfying the Leibnitz rule with respect to the multiplication on $\bigoplus B_i$. 
Then the commutator 
\begin{equation}\label{ksclassintrobis}
[\nabla ', d] \in \prod Hom_R (B_i, B_{i+1}) \otimes \Omega^1_R 
\end{equation}
with the differential $d$ on $\Bt$ 
commutes with $d$ and it is a $R$-linear derivation of $\Bt$  with values in  $\Bt  \otimes \Omega^1_R $  of degree $1$. Any such derivation determines a class in
$$HH^2(\Bt, \Bt)\otimes _R \Omega^1_R \iso HH^2(\At, \At)\otimes _R \Omega^1_R. $$
The class corresponding to $[\nabla ', d] $ is independent of the choices we made. This is the Kodaira-Spencer class  (\ref{ksclassintro})\footnote{The Kodaira-Spencer class of a 
$A_{\infty}$-algebra is defined in  \cite{ge}, \S 3. Our exposition is inspired by \cite{lo}.}.

The Kodaira-Spencer class (\ref{ksclassintro}) acts on the Hochschild homology:
$$e_{\kappa}: HH_{\idot}(\At, \At) \to HH_{\idot -2}(\At, \At)  \otimes _R \Omega^1_R.$$
 The operator $e_{\kappa}$ is  induced by the action of the Hochschild cohomology algebra on the Hochschild homology (referred to as the  ``interior product'' action).

\subsection{The Hodge filtration on the periodic cyclic homology.}  Denote by $(\Ct(\At, \At), b)$ the relative Hochschild chain complex of $\At$  and by $\CPt (\At)= (\Ct (\At, \At)((u)), b +uB)$ the periodic cyclic complex.
The complex $\CPt (\At)$ is equipped with 
the Hodge filtration $$\cF^i\CPt (\At):= (u^i \Ct (\At, \At)[[u]], b +uB),$$
which induces a Hodge filtration $\cF^{\idot} HP_{\idot} (\At)$ on the periodic cyclic homology and  
a convergent Hodge-to-de Rham spectral sequence 
\begin{equation}\label{hdrss}
HH_{\idot}(\At, \At)((u))  \Rightarrow HP_\idot(\At).
\end{equation}
The Gauss-Manin connection $\nabla$ on the periodic cyclic homology (we recall its construction in \S   \ref{sGMonHP})
satisfies the Griffiths transversality condition 
$$\nabla:  \cF^{\idot}HP_{\idot} (\At) \to \cF^{\idot -1}HP_{\idot} (\At) \otimes _R \Omega^1_R.$$

\subsection{Statement of the main result.} Recall that $\At$ is called homologically proper  if $\At$ is perfect as a complex of $R$-modules.
 A DG algebra $\At$ is said to be  homologically smooth if  $\At$  is quasi-isomorphic
to a DG algebra $\Bt$, which is $h$-flat as a complex of $R$-modules\footnote{A  complex $\Bt$ of $R$-modules is called $h$-flat if, for any acyclic complex  $\Dt$ of $R$-modules, the tensor product $\Bt \otimes _R \Dt$ is acyclic.},  and $\Bt$  is perfect as a DG module over $ \Bt \otimes _R \Bt ^{op}$.
The following is one of the main results of our paper.
\begin{Th}\label{fmth} Fix the pair $(\tR, \tF)$ as in \S \ref{s.s.r.f.l.m.}.
% and assume, in addition, that the characteristic $p$ of $k$ is odd. 
 Let $\At$ be a homologically smooth and homologically
proper  DG algebra over $R$ such that 
\begin{equation}\label{dimbound}
HH_m(\At, \At)=0, \;  \text{for every m with}  \; |m| \geq p-2.
\end{equation}
 Then a lifting\footnote{A lifting of $\At$ over $\tR$ is an $h$-flat  DG algebra 
$\tilde{\At}$ over $\tR$  together with a quasi-isomorphism 
$\tilde{\At}\otimes_{\tR} R \iso \At $ of
DG algebras over $R$.} of $\At$ over $\tR$, if it exists,  specifies an isomorphism
\begin{equation}\label{frobeniusonperidoc}
\phi: \cC_{(\tR, \tF)}^{-1} (\Gr_\cF^{\idot} HP_\idot (\At)  , \Gr_\cF {\nabla}) \iso (HP_\idot (\At), \nabla)
\end{equation}
giving  $(HP_\idot (\At), \nabla, \cF^{\idot} HP_\idot (\At))$ a Fontaine-Laffaille module structure. In addition, the Hodge-to-de Rham spectral sequence (\ref{hdrss}) degenerates at $E_1$ and induces an isomorphism of Higgs modules
\begin{equation}\label{higgshodgeiso}
(\Gr_\cF^{\idot} HP_\idot (\At)  , \Gr_\cF {\nabla})\iso (HH_{\idot}(\At, \At)[u, u^{-1}], u^{-1} e_{\kappa}). 
\end{equation}
\end{Th}
Using (\ref{higgshodgeiso}),  the isomorphism (\ref{frobeniusonperidoc}) takes the form
\begin{equation}\label{frobeniusonperidocbis}
\phi:  (F^*HH_{\idot}(\At, \At)[u, u^{-1}], \nabla_{can} +  u^{-1}C^{-1}_{(\tR, \tF)}(e_{\kappa}))   \iso (HP_\idot (\At), \nabla),
\end{equation}
where $\nabla_{can}$ is the Frobenius pullback connection and $C^{-1}_{(\tR, \tF)}$ is the inverse Cartier operator (\ref{inversecartieroperator}). 
\begin{remarks}
\begin{enumerate}[(a)]
\item If $R=k$ the above result, under slightly different assumptions\footnote{Kaledin proves his result assuming, instead of (\ref{dimbound}),  vanishing of the {\it reduced Hochschild cohomology} 
$\overline{HH}^m(\At)$ for $m\geq 2p$.}, is due to Kaledin (\cite{ka3}, Th. 5.1). 
\item The construction from Theorem  \ref{fmth} determines a functor from the category of homologically smooth and homologically
proper  DG algebras over $\tR$ satisfying (\ref{dimbound}) localized with respect to quasi-isomorphisms to the category of Fontaine-Laffaille modules. We expect, but do not check it in this paper, that
this functor extends to the homotopy category of smooth and proper DG categories over $\tR$ satisfying the analogue of (\ref{dimbound}). 
When applied to the bounded derived DG category $D^b(Coh(\tX))$ of coherent sheaves on a smooth proper scheme $\tX$ over $\tR$ of relative dimension  less than $p-2$, we expect to recover
the Fontaine-Laffaille structure on  
$$HP_0(D^b(Coh(X))\iso \bigoplus_i H^{2i}_{DR}(X)(i)$$
$$HP_1(D^b(Coh(X))\iso \bigoplus_i H^{2i+1}_{DR}(X)(i)$$
constructed by Faltings  in (\cite{fa}, Th. 6.2).
Here $X$ denotes the scheme over $R$ obtained from $\tX$ by the base change and $H^{*}_{DR}(X)(i)$ the Tate twist of the Fontaine-Laffaille structure on the relative de Rham cohomology.
\end{enumerate}
\end{remarks}
Let us explain some corollaries of Theorem \ref{fmth}.  First, under the assumptions of Theorem \ref{fmth}  the Hochschild and cyclic homology of $\At$ is a locally free  $R$-module. 
This follows from a general property of Fontaine-Laffaille modules mentioned above. Next, it follows, that under the same assumptions the $p$-curvature of the Gauss-Manin 
connection on $HP_{\idot}(\At)$ is nilpotent\footnote{This suffices for our main application in characteristic $0$: Theorem \ref{lmth} below.}. In fact, there is a decreasing filtration, 
\begin{equation}\label{conjfiltb}
\cV_{i} HP_\idot (\At) \subset HP_\idot (\At)
\end{equation}
formed by the images under $\phi$ of 
 $$u^i F^*HH_{\idot}(\At, \At)[u^{-1}] \subset F^*HH_{\idot}(\At, \At)[u, u^{-1}]$$
 which is preserved by the connection and such that $\Gr^\cV_{\idot} HP_\idot (\At)$ has zero $p$-curvature:
  \begin{equation}\label{cartierisomorphism}
(\Gr^\cV_{\idot} HP_\idot (\At),  \Gr^\cV {\nabla}) \iso (F^*HH_{\idot}( \At, \At)[u, u^{-1}], \nabla_{can}).
\end{equation}
 Moreover, using Theorem \ref{fmth} we can express the $p$-curvature of $\nabla$ on $HP_{\idot}(\At)$ in terms of the Kodaira-Spencer operator $e_{\kappa}$: by (\cite{ov}, Th. 2.8),
for any Higgs module $(E, \theta)$, such that the action of $S^p T_R$  on $E$ is trivial, the $p$-curvature of 
$\cC_{(\tR, \tF)}^{-1}(E, \theta)$, viewed as a $R$-linear morphism
$$\psi: F^*E\to F^*E \otimes F^* \Omega^1_R$$
is equal to $-F^*(\theta)$. In particular, under assumption (\ref{dimbound}),  
the $p$-curvature of  $\cC_{(\tR, \tF)}^{-1}(HH_{\idot}(\At, \At)[u, u^{-1}], u^{-1} e_{\kappa})$, 
equals $-u^{-1} F^* (e_{\kappa})$.
 As a corollary,  we obtain, a version of the Katz formula for the $p$-curvature of the Gauss-Manin connection on the de Rham cohomology (\cite{katz2}, Th. 3.2):
 by (\ref{cartierisomorphism}) the $p$-curvature morphism for $ HP_\idot (\At) $ shifts the filtration $\cV_{\idot}$: 
 $$\psi: \cV_{\idot} HP_\idot (\At)\to \cV_{\idot -1} HP_\idot (\At)\otimes F^* \Omega^1_R.$$
 Thus, $\psi$ induces a morphism
$$\overline{\psi}:  \Gr^\cV_{\idot} HP_\idot (\At)\to \Gr^\cV_{\idot -1} HP_\idot (\At)\otimes F^* \Omega^1_R.$$
Our version of the Katz formula asserts the commutativity of   
the following diagram.  
  \begin{equation}\label{katzfontaine}
\def\normalbaselines{\baselineskip20pt
\lineskip3pt  \lineskiplimit3pt}
\def\mapright#1{\smash{
\mathop{\to}\limits^{#1}}}
\def\mapdown#1{\Big\downarrow\rlap
{$\vcenter{\hbox{$\scriptstyle#1$}}$}}
\begin{matrix}
 \Gr^\cV_{i} HP_j (\At) & \iso  &F^*HH_{j+2i}( \At, \At)  \cr
 \mapdown{\overline{\psi}}  && \mapdown{-F^*(e_{\kappa} ) }  \cr
\Gr^\cV_{i -1} HP_j (\At) \otimes F^*\Omega^1_R  & \iso & F^*HH_{j+2i-2}( \At, \At) \otimes F^*\Omega^1_R .
\end{matrix}
 \end{equation}

\subsection{The co-periodic cyclic homology, the conjugate filtration, and a generalized Katz $p$-curvature formula.} 
Though, as explained above,  formula (\ref{katzfontaine}) is an immediate corollary of Theorem  \ref{fmth}, a version of the  former holds for any DG algebra $\At$. What makes this generalization
possible is the observation that although  
the morphism  (\ref{frobeniusonperidocbis}) does depend on the choice of a lifting of $\At$ over $\tR$ the induced $\nabla$-invariant filtration (\ref{conjfiltb}) is canonical: in fact, it coincides with the {\it conjugate filtration}  introduced by Kaledin in \cite{cop}.\footnote{The terminology is borrowed from \cite{katz1}, where the conjugate filtration on the de Rham cohomology in characteristic $p$ was
introduced.}   However, in general, the conjugate filtration is a filtration on the {\it co-periodic cyclic homology} $\overline{HP}_\idot(\At)$ defined Kaledin in  {\it loc. cit.} to be the homology of the complex  
 $$\overline{\CPt}(\At) = (\Ct (\At, \At)((u^{-1})), b +uB).$$
For any $\At$, this comes together with the  conjugate filtration $\cV_{\idot} \overline{CP}_\idot(\At)$ satisfying the properties 
$$u: \cV_{\idot} \overline{CP}_\idot(\At)\iso \cV_{\idot+1} \overline{CP}_{\idot}(\At)[2],$$
$$ \Gr^\cV \overline{CP}_\idot(\At) \iso F^*C( \At, \At)((u^{-1})).$$  
This yields a convergent {\it conjugate spectral sequence}  
\begin{equation}\label{conjss}
F^*HH_{\idot}( \At, \At)((u^{-1}))  \Rightarrow \overline{HP}_\idot (\At),
\end{equation}
whose $E_{\infty}$ page is $\Gr^\cV_{\idot} \overline{HP}_{\idot} (\At)$.
It is shown in \cite{cop} that if $\At$ is smooth and homologically bounded then 
the morphisms 
\begin{equation}\label{finitepercomp}
 (\Ct (\At, \At)[u, u^{-1}], b +uB)\rar{}  (\Ct (\At, \At)((u)), b +uB)
 \end{equation}
 \begin{equation}\label{finitecopercomp}
(\Ct (\At, \At)[u, u^{-1}], b +uB) \rar{}  (\Ct (\At, \At)((u^{-1})), b +uB)
\end{equation}
are quasi-isomorphisms. In particular,  for smooth and homologically bounded DG algebras one has a canonical isomorphism
\begin{equation}\label{periodiccopercomp}
\overline{HP}_{\idot} (\At) \iso HP_{\idot} (\At).
 \end{equation}
 For an arbitrary  DG algebra $\At$ we introduce in \S \ref{sGMonHP} a Gauss-Manin connection on $\overline{HP}_\idot(\At)$. It is compatible with the one on $HP_{\idot} (\At)$ if $\At$ is smooth and homologically bounded.
 We show that $\nabla$ preserves the conjugate filtration and
 the entire conjugate spectral sequence (\ref{conjss}) is compatible with the connection (where the first page, $F^*HH_{\idot}( \At, \At)((u^{-1}))$ is endowed with the Frobenius pullback connection). 
 In particular, the   $p$-curvature $\psi$ of the connection on  $\overline{HP}_\idot(\At)$ is zero on $\Gr^\cV_{\idot} \overline{HP}_{\idot} (\At)$. Hence, $\psi$ induces a morphism
$$\overline \psi:  \Gr^\cV_{\idot} \overline{HP}_{\idot} (\At) \to \Gr^\cV_{\idot -1} \overline{HP}_{\idot} (\At) \otimes F^* \Omega^1_R .$$
In \S \ref{sGMonHP} we prove the following result, which is a generalization of formula (\ref{katzfontaine}).
\begin{Th}\label{katzformulaspectral} Let $\At$ be a DG algebra over $R$ and $\kappa \in HH^2(\At, \At)\otimes _R \Omega^1_R$  its Kodaira-Spencer class.
\begin{enumerate}[(a)]
\item The morphism  $u^{-1} F^* (e_{\kappa}):  F^*HH_{\idot}( \At, \At)((u^{-1})) \to F^*HH_{\idot}( \At, \At)((u^{-1})) \otimes F^* \Omega^1_R$ commutes with all the differentials in the conjugate spectral
sequence  (\ref{conjss}) inducing a map
$$ \Gr^\cV_{\idot} \overline{HP}_{\idot} (\At) \to \Gr^\cV_{\idot -1} \overline{HP}_{\idot} (\At) \otimes F^* \Omega^1_R,$$
which we also denote by $u^{-1} F^* (e_{\kappa})$.
With this notation, we have 
\begin{equation}\label{katzgeneral}
 u^{-1} F^* (e_{\kappa}) = \overline \psi .
 \end{equation}
\item
Assume that $HH_m(\At, \At)=0$ for all sufficiently negative  $m$. Then the $p$-curvature of  the Gauss-Manin connection on $\overline{HP}_\idot(\At)$ is nilpotent.
\end{enumerate}
\end{Th}
 \begin{cor}\label{nilpotencygeneral}
 Let $\At$ be a smooth and proper DG algebra over $R$ and let $d$ be a non-negative integer $d$ such that $HH_m(\At, \At)=0$, for every $m$ with $|m|>d$. 
Then the $p$-curvature of  the Gauss-Manin connection on $HP_\idot(\At)$ is nilpotent of exponent $\leq d+1$, {\it i.e.}, there  exists 
a filtration $$0=\cV_0  HP_\idot(\At) \subset \cdots \subset \cV_{d+1} HP_\idot(\At) =  HP_\idot(\At)$$ preserved by the connection such that, for every $0<i \leq d+1$, the $p$-curvature of the connection on
$\cV_i/\cV_{i-1}$ is $0$.
\end{cor}
\subsection{An application: the local monodromy theorem.} 
As an application of the nilpotency of the $p$-curvature 
we prove, using a result from (\cite{katz1}),  ``the local monodromy theorem'' for the periodic cyclic homology in characteristic $0$.
\begin{Th}\label{lmth}
 Let $S$ be a smooth irreducible affine curve over  $\bC$, 
 $\Sb$ a smooth compactification of $S$, and let
 $\At$ be a smooth and proper  DG algebra over $\cO(S)$.  Then the Gauss-Manin connection on the relative periodic cyclic homology $HP_* (\At)$ has regular  singularities,  
 and its monodromy around every point at  $\Sb -S$ is quasi-unipotent. 
\end{Th}
This result generalizes the Griffiths-Landman-Grothendieck theorem asserting that for a smooth proper scheme $X$ over $S$  
the Gauss-Manin connection on the relative de Rham cohomology $H_{DR}^*(X)$ has regular  singularities,  
 and  its monodromy at infinity is quasi-unipotent. The derivation of Theorem \ref{lmth} from Corollary \ref{nilpotencygeneral} is essentially due to Katz (\cite{katz1}); we explain the argument  in 
 \S \ref{SLMT}. 

\subsection{Proofs.} 
Let us outline the proofs of Theorems \ref{fmth} and \ref{katzformulaspectral}.  Without loss of generality we may assume that $\At$ is a semi-free DG algebra over $R$.
Let $\At^{\otimes p}$ denote the $p$-th tensor power of $\At$ over $R$. This is a DG algebra equipped with an action of the symmetric group $S_p$. In particular, it carries an action of the  group
$\bZ/p\bZ \iso C_p \subset S_p$ of cyclic permutations. We denote by $T(C_p,  \At^{\otimes p})$ the Tate cohomology complex of $C_p$ with coefficients in $\At^{\otimes p}$.
 The algebra structure on $\At^{\otimes p}$ induces one on the Tate cohomology $\vH^{\idot}(C_p,  \At^{\otimes p})$. Moreover, choosing an appropriate ``complete resolution'' one can  lift the cup product on the cochain level giving $T(C_p,  \At^{\otimes p})$ the structure of a DG algebra over $R$.  If $\At = A$ is an associative algebra then, for $p\ne 2$,  one has a canonical isomorphism
 of algebras
 $$\vH^{*}(C_p,  A^{\otimes p}) \iso F^*A  \otimes \vH^{*}(C_p,  \bF_p)  \iso  F^*A [u, u^{-1}, \epsilon ],$$
 $\deg u =2$ and $\deg \epsilon =1$, $\epsilon ^2 =0$.
 In general, Kaledin defines an increasing filtration 
 $$\tau _{\leq \idot}^{dec}T(C_p, \At^{\otimes p}) \subset  T(C_p, \At^{\otimes p})$$
 making $T(C_p, \At^{\otimes p})$ a filtered DG algebra equipped with a canonical quasi-isomorphism of graded DG algebras
\begin{equation}\label{fitltdecgr}
\bigoplus_i \Gr^{\tau}_{i}T(C_p, \At^{\otimes p})\iso F^*\At \otimes \bigoplus_i   \vH^{i}(C_p,  \bF_p)[-i].
\end{equation}
Note that the right-hand side  of (\ref{fitltdecgr})  has a canonical connection - the Frobenius pullback connection. A key observation explained in \S \ref{ConT} is that there is a canonical connection $\nabla$
on the filtered DG algebra $T(C_p, \At^{\otimes p})$, which induces the Frobenius pullback connection on $\Gr^\tau$.

Denote by  $T_{[m, l]}(C_p, \At^{\otimes p})$, ($m\leq l$), the quotient of $\tau _{\leq l}^{dec}T(C_p, \At^{\otimes p})$ by $\tau _{\leq m-1}^{dec}T(C_p, \At^{\otimes p})$. The DG algebra 
$$\cB(\At):= T_{[-1, 0]}(C_p, \At^{\otimes p}),$$
 which is a square-zero extension of $F^*\At$
 $$ F^*\At[1] \rar{\mu} \cB(\At) \rar{} F^*\At$$  
  with a compatible connection $\nabla$,  admits another description. Let $\hat R$ be a flat lifting of $\tR$ over $W(k)$, $\hat i_*$ the functor from the category of DG algebras over $R$ to the category of DG algebras over $\hat R$, which carries a DG algebra over $R$ to the same underlying DG ring with the action of $\hat R$ induced by the morphism $\hat R \to R$, and let $L\hat i^*$ be the left adjoint functor, which carries a DG algebra $C_{\idot}$ over $\hat R$ to the derived tensor product $C_{\idot}\stackrel{L}{\otimes}_{\hat R} R$. For any 
DG algebra $\At$ over $R$ the composition $L\hat i^*  \hat i_* \At$ is an algebra over   $L\hat i^*  \hat i_* R \iso R[\mu]$, where $\deg \mu =-1$, $\mu ^2 =0$.
One can easily check that the functor $L\hat i^*  \hat i_*$ depends on $\tR$ only (in particular, every automorphism of $\hat R$, which restricts to the identity on $\tR$ acts trivially on $L\hat i^*  \hat i_*$). 
Similarly, the morphism of crystalline toposes $\mathrm{Cris}(R/k)\to \mathrm{Cris}(R/W(k))$ induces a functor $\hat i_{*cris}$ from the category of DG algebras in the category of crystals on
 $\mathrm{Cris}(R/k)$ ({\it i.e.}, the category of $R$-modules with integrable connections) to the category of DG algebras in the category of crystals on
 $\mathrm{Cris}(R/W(k))$ ({\it i.e.}, the category of $p$-adically complete $\hat R$-modules with integrable connections) and also its left adjoint functor $L\hat i^{*cris}$. These functors extend the functors $\hat i_*, L \hat i^*$ on sheaves of rings\footnote{Recall that one should think of an algebra with connection on a scheme $X$ over a base ring $R$ as a family of algebras with connection over points of $\Spec(R)$. In terms of this picture, the pushforward and pullback functors $i_{*cris}, i^{*cris}$ are just base change.}.
\begin{Th}\label{cBandLift}
\begin{enumerate}[(a)] Let $\At$ be  a term-wise flat DG algebra over $R$\footnote{Since $R$ has finite homological dimension $\At$ is also $h$-flat over $R$.}. 
\item 
 There is a canonical quasi-isomorphism of DG algebras with connection
 $$(\cB(\At), \nabla )   \iso  L\hat i^{* cris}\hat i_{* cris} F^*\At .$$
\item 
A lifting  $(\tR, \tF)$ of $(R, F)$ over $W_2(k)$ gives rise to a canonical quasi-isomorphism of DG algebras with connection
 $$(\cB(\At), \nabla )  \iso  \cC_{(\tR, \tF)}^{-1} (L\hat i^*  \hat i_* \At, \mu \tilde \kappa).$$
Here   $\tilde \kappa $ is the Kodaira-Spencer class of $\At$ regarded as a derivation of $\At$ with values in $\At \otimes \Omega^1_R$ of degree  $1$ (as defined by formula
\ref{ksclassintrobis}  ),  $\mu \tilde \kappa$ the induced degree $0$ derivation of $L\hat i^*  \hat i_* \At$ with values in $(L\hat i^*  \hat i_* \At) \otimes  \Omega^1_R$  ,  and $ \cC_{(\tR, \tF)}^{-1}$ is the inverse Cartier transform.
\item
A lifting of $\At$ over $\tR$ gives rise to a canonical quasi-isomorphism of DG algebras with connection
$$(\cB(\At), \nabla)\iso \cC_{(\tR, \tF)}^{-1} (\At [\mu], \mu \tilde \kappa).$$
\end{enumerate}
\end{Th}
\begin{remarks}
\begin{enumerate}[(a)]
\item If $R$ is a perfect field the above result is due to Kaledin (\cite{Bokst}, Prop. 6.13).
\item The first part of the Theorem together with the projection formula gives a canonical isomorphism
of DG algebras with connections
$$\hat i^{cris}_*\cB(\At) \iso  \hat i^{cris}_* F^*\At  \oplus \hat i^{cris}_* F^*\At [1],$$
where the right-hand side of the equation is the trivial square-zero extension with the Frobenius pullback connection.
However, in general $\cB(\At)$ does not split. For example, from the second part of the Theorem it follows that the $p$-curvature of $\nabla$ on  $\cB(\At)$ equals  $- \mu F^* (\mu \tilde \kappa )$. In particular, it is not zero as long as $\tilde \kappa$ is not $0$.
\end{enumerate}
\end{remarks}
Next, we relate the cyclic homology of $\cB(\At)$ together with the connection induced by the one on $\cB(\At)$ with the periodic cyclic homology of $\At$ with the Gauss-Manin connection.
The two-step fitration $ F^*\At [1] \subset \cB(\At)  $ gives rise to a  filtration $\cV_{m}CC(\cB(\At) )\subset CC(\cB(\At) ) $, $(m=0, -1, -2, \cdots )$, on the cyclic complex of  $\cB(\At)$. 
\begin{Th}\label{cBandCP} Let $\At$ be a term-wise flat DG algebra over $R$. 
We have a canonical quasi-isomorphism of filtered complexes with connections
$$\cV_{[-p+2, -1]}CC(\cB(\At))[1]\iso \cV_{[-p+2, -1]}\overline{CP}(\At). $$
Moreover, the multiplication by $u^{-1}$ on the right-hand side corresponds under the above quasi-isomorphism to the multiplication by the class $B\mu$ in the second negative cyclic homology group of the algebra
$k[\mu]$. 
\end{Th}
Let us derive Theorem \ref{fmth} from Theorems \ref{cBandCP} and \ref{cBandLift}. Since the Cartier transform is a monoidal functor, we have by part 3 of Theorem \ref{cBandCP}
$$ (\cV_{[-p+2, -1]}CC(\cB(\At)),  \nabla )\iso \cC_{(\tR, \tF)}^{-1} (\cV_{[-p+2, -1]}CC (\At [\mu]),  \mu \tilde \kappa).$$
We compute the right-hand side using the K{\"u}nneth formula: with obvious notation we have a quasi-isomorphism of mixed complexes
$$\cV_{[-p+2, -1]}C (\At [\mu])\iso C(\At) \otimes \cV_{[-p+2, -1]}C(k[\mu]).$$ 
 The Hochschild  complex of $k[\mu]$ regarded as a mixed complex is quasi-isomorphic to the divided power algebra: 
$$C(k[\mu], k[\mu])\iso k\langle \mu, B\mu \rangle$$
with zero differential and Connes' operator acting by the formulas: $B( (B\mu) ^{[m]})=0$, $B( \mu (B\mu) ^{[m]})=(m+1)(B\mu) ^{[m+1]}$.

  It follows that
$$\cV_{[-p+2, -1]}CC (\At [\mu])\iso \bigoplus_{0\leq m \leq p-3} C(\At) \otimes \mu (B\mu)^{[m]}.$$
Setting $B\mu =u^{-1}$ and using the Cartan formula (\cite{ge}; see also \S \ref{sGMonHP} for a review), we find
$$(\cV_{[-p+2, -1]}CC (\At [\mu],  \mu \tilde \kappa)[-1] \iso (C(\At) \otimes k[u^{-1}]/u^{2-p}, u^{-1}\iota_{\tilde \kappa}).$$
Summarizing, we get
$$(\cV_{[-p+2, -1]}\overline{CP}(\At),  \nabla ) \iso \cC_{(\tR, \tF)}^{-1} (C(\At) \otimes k[u^{-1}]/u^{2-p}, u^{-1}\iota_{\tilde \kappa})[2]$$
This implies the desired result. The derivation of Theorem \ref{katzformulaspectral} is similar.

\subsection*{Acknowledgments}
The authors would like to express their gratitude to D. ~Kaledin for teaching them most of the
mathematics that went into this paper. We are also extremely grateful to the referee 
 for his detailed and helpful
comments, which have greatly helped to improve the exposition. 

A.P. was partially supported by the Russian Academic Excellence Project '5-100' and by Dobrushin stipend.

D.V.  was partially supported by the National Science Foundation Graduate Research Fellowship under Grant No. DMS-1000122015

V.V was partially supported by Laboratory of Mirror Symmetry NRUHSE, RF government grant, ag. 
numbers  \textnumero  $  14.641.31.0001.$

\section{The Tate cohomology complex of $A_{\cdot} ^{\otimes  p} $ }
In this section we construct a connection on the Tate complex $T(C_p, \At ^{\otimes p})$ and prove Theorem \ref{cBandLift}.
\subsection{The Tate cohomology complex.}\label{sstcc}
Let $G$ be a finite group. A complete resolution of the trivial $\bZ[G]$-module $\bZ$ is an acyclic complex of free $\bZ[G]$-modules
$$ \rar{} \cdots P_{-2} \rar{} P_{-1}\rar{} P_{0}\rar{} P_{1} \rar{} P_{2} \rar{} \cdots $$
together with an isomorphism  of $\bZ[G]$-modules 
$$\epsilon: \bZ \iso \ker (d: P_{0}\rar{} P_{1}).$$
One can show that for any two complete resolutions $(\Pt, \epsilon)$, $(\Pt ^{\prime}, \epsilon ')$ there exists a morphism $f_{\idot}:\Pt \to \Pt ^{ \prime }$ of complexes of $\bZ[G]$-modules such that $f_0\circ \epsilon = \epsilon '$ and
such $f_{\idot}$ is unique up to homotopy (in fact,  $Hom(\Pt , \Pt ^{\prime })$ in the homotopy category $Ho(\bZ[G])$ of complexes of  $\bZ[G]$-modules is canonically isomorphic to $\bZ/r\bZ$, where $r$ is the order of $|G|$ : to see this observe that $Hom_{Ho(\bZ[G])}(\Mt , \Pt^{\prime })=0$ if $\Mt$ is either a bounded from above complex of free  $\bZ[G]$-modules or a bounded from below acyclic complex. It follows, using the canonical and ``stupid'' truncations,  that $Hom_{Ho(\bZ[G])}(\Pt, \Pt^{\prime })\simeq
Hom_{Ho(\bZ[G])}(\bZ, \Pt ^{\prime }) \simeq \hat H ^0(G, \bZ)\simeq    \bZ/r \bZ$ ).  Fix  a complete resolution  $(\Pt, \epsilon)$.  For a complex $\Mt$ of $\bZ[G]$-modules
we define its Tate cohomology complex $T(G, \Mt)$ to be
$$ T(G, \Mt): = (\Mt \otimes _\bZ \Pt )^G.$$
This defines a DG functor $T(G, \cdot)$ from the DG category $C(Mod(\bZ[G]))$ of complexes of $\bZ[G]$-modules to the DG category of complexes of abelian groups. By construction, $T(G, \cdot)$ commutes with arbitrary direct sums. 
Also, it easy to check that   $T(G, \cdot)$ carries bounded  complexes of free   $\bZ[G]$-modules and bounded acyclic complexes to acyclic complexes.\footnote{Both statements may fail for unbounded complexes. For example,  $\epsilon$ induces a quasi-isomorphism
$T(G, \bZ) \iso T(G,\Pt)$. Thus,   $T(G, \cdot)$ does not respect arbitrary quasi-isomorphisms.} We denote the cohomology groups of $T(G, \Mt)$ by $\vH^{*}(G, \Mt).$

A {\sf  multiplicative}  complete resolution  is  a complete resolution    $(\Pt, \epsilon)$ together with a DG ring structure  
$$m:  \Pt \otimes \Pt  \to \Pt $$
which is compatible with the $G$-action ({\it i.e.}, $m$ is a morphism of complexes of $\bZ[G]$-modules) such that $\epsilon: \bZ \to  \Pt $ is a morphism of DG rings.
Multiplicative complete resolutions exist:  {\it e.g.},  see   \cite{cf}, Chapter 4, \S 7. From now on $ T(G, \Mt)$ will denote the Tate complex associated with a  fixed multiplicative  complete resolution. Then, for any complexes  $\Mt$, $\Mt'$ of $\bZ[G]$-modules, we get a natural morphism
$$ T(G, \Mt) \otimes T(G, \Mt ')     \to  T(G, \Mt \otimes _\bZ  \Mt ' ),$$
which induces the cup product on the Tate cohomology groups.
In particular, if $\Mt$ is a DG ring with an action of $G$, then the Tate complex $ T(G, \Mt)$ acquires a DG ring structure.

\subsection{The functor $V \mapsto T(C_p, V^{\otimes ^p})$}\label{tatefilt}

Let $R$ be a finitely generated smooth commutative  algebra over a perfect field $k$ of characteristic $p>0$. 
 For any complex $\Vt$ of flat $R$-modules the tensor power $V^{\otimes ^ p}$ over $R$ carries an action of the symmetric group $S_p$. Denote by $\bZ/p\bZ\simeq C_p \subset S_p$ the subgroup of cyclic permutations.  We consider  the functor $T(C_p,\cdot^{\otimes p})$ 
 $$\Vt \mapsto T(C_p,\Vt^{\otimes p})$$ from the category of complexes of flat $R$-modules to the category of complexes of all $R$-modules. This functor has a number of remarkable properties not obvious from the definition. First, if $\Vt=V$ is supported in cohomological degree $0$ and $p>2$, then,  we have a canonical isomorphism of  graded modules over $R\otimes \vH^{*}(C_p, \bF_p)$
 $$\vH^{*}(C_p,  V^{\otimes p}) \iso F^*V  \otimes_{\bF_p}  \vH^{*}(C_p,  \bF_p) \iso  F^*V [u, u^{-1}, \epsilon ],$$
where  $\deg u =2$ and $\deg \epsilon =1$, $\epsilon ^2 =0$.
In  \cite{Bokst}, \S 6.2, Kaledin generalized the above isomorphism: for every complex $\Vt$ of flat $R$-modules and any $p$ he defines a canonical increasing filtration 
 $$\tau _{\leq \idot}^{dec}T(C_p, \Vt^{\otimes p}) \subset  T(C_p, \Vt^{\otimes p})$$
 making $T(C_p, \Vt^{\otimes p})$ a filtered DG module over  $T(C_p, \bF_p)$ (endowed with the canonical filtration) equipped with a canonical quasi-isomorphism of graded DG modules over the graded DG algebra $R\otimes \bigoplus_i  \vH^{i}(C_p,  \bF_p)[-i]$ (\cite{Bokst}, Lemma 6.5)
\begin{equation}\label{fitltdecgrv}
\bigoplus_i \Gr^{\tau}_{i}T(C_p, \Vt^{\otimes p})\iso \bigoplus_i  (F^*\Vt \otimes \vH^{i}(C_p,  \bF_p)[-i]).
\end{equation}
Namely,  consider the (decreasing) stupid filtration on   $\Vt^{\otimes p}=\dots \to (\Vt^{\otimes p})_{i}\to (\Vt^{\otimes p})_{i+1}\to \dots$   rescaled by $p$: $$F^i\Vt^{\otimes p}=\dots\to 0\to (\Vt^{\otimes p})_{ip}\to (\Vt^{\otimes p})_{ip+1}\dots $$
It induces a filtration $F^{\idot}$ on $T(C_p,\Vt^{\otimes p})$. Now, we apply to the filtered complex $T(C_p,\Vt^{\otimes p})$  Deligne's "filtered truncation" construction(\S1.3.3 in \cite{d}) and define 
%\begin{multline}
$$ \tau _{\leq n}^{dec}T(C_p,\Vt^{\otimes p})_i= F^{i-n}T(C_p,\Vt^{\otimes p})_i \cap d^{-1}(F^{i+1-n}T(C_p,\Vt^{\otimes p})_{i+1}) $$
%\end{multilane}
We denote by $T_{[n, m]}(C_p, \Vt^{\otimes p})$, ($n \leq m$), the quotient of $\tau _{\leq m}^{dec}T(C_p, \Vt^{\otimes p})$ by $\tau _{\leq n-1}^{dec}T(C_p, \Vt^{\otimes p})$,
and by $\overline{T}(C_p,\Vt^{\otimes p})$ the completion of $T(C_p,\Vt^{\otimes p})$ with respect to the filtration $ \tau _{\leq n}^{dec}$:
$$\overline{T}(C_p,\Vt^{\otimes p})=  \varprojlim T(C_p,\Vt^{\otimes p})/  \tau _{\leq n}^{dec}T(C_p,\Vt^{\otimes p}).$$
Formula (\ref{fitltdecgrv}) implies the following surprising result.
\begin{lm}[cf. {{\cite{lu}, Proposition 2.2.3}}, see also \S {\rm III}.1.1 in \cite{ns}]\label{tateexact} The functors  $T_{[n,m]}(C_p, \cdot ^{\otimes ^ p})$, ($n\leq m$), and  $\overline{T}(C_p,\cdot ^{\otimes ^ p})$
carry an acyclic complex of flat $R$-modules to an acyclic complex. Moreover,
 $T_{[n,m]}(C_p, \cdot ^{\otimes ^ p})$ and $\overline{T}(C_p,\cdot ^{\otimes ^ p})$ are exact:
if $X\to Y\to Z\to X[1]$ is a triangle of complexes of flat $R$-modules which is distinguished in $D(R)$ then the total complex of the double complexes 
$T_{[n,m]}(C_p,X^{\otimes ^p})\to T_{[n,m]}(C_p, Y^{\otimes ^p})\to T_{[n,m]}(C_p, Z^{\otimes ^ p})$ and
$\overline{T}(C_p,X^{\otimes ^p})\to \overline{T}(C_p, Y^{\otimes ^p})\to \overline{T}(C_p, Z^{\otimes ^ p})$ are acyclic.
\end{lm}
\begin{proof}
 Using formula (\ref{fitltdecgrv}), we get a functorial quasi-isomorphism
  $$T_{[n,n]}(C_p, \Vt  ^{\otimes ^ p})\iso F^*\Vt[-n].$$ This proves the Lemma for $T_{[n,n]}(C_p, \cdot ^{\otimes ^ p})$. The general case follows by  d\'evissage.
\end{proof}
\begin{rem} 
The proof of Lemma \ref{tateexact} given here is due to Kaledin. We refer  the reader to ({\cite{lu}, Proposition 2.2.3}) and to  \S {\rm III}.1.1 in \cite{ns} for different proofs of similar statements and for the historical background. 
Using finiteness of homological dimension of $R$ one can adapt the argument from \S {\rm III}.1.1 in \cite{ns} to prove that $T(C_p, \cdot)$ has the properties as in Lemma \ref{tateexact}. We will not use this fact.
\end{rem}\label{infinite}
The filtration $\tau _{\leq \idot}^{dec}T(C_p,\cdot^{\otimes p})$ is compatible with the cup product in the obvious sense. In particular,  
if $\At$ is a termwise flat DG algebra over $R$  then the filtration $$\tau _{\leq \idot}^{dec}T(C_p,\At^{\otimes p})$$ defines the structure of a filtered DG algebra on $T(C_p,\At^{\otimes ^p})$ and on 
$\overline{T}(C_p,\At^{\otimes p})$.

\subsection{Connection on the Tate complex} \label{ConT}

Denote $spec\, R$ by $X$. 
  Let  $X^{[2]}$ be the first infinitesimal thickening of the diagonal $\Delta\subset X\times X$ and $p_1,p_2:X^{[2]}\to X$ projections. The following construction is essentially contained in \cite{ka1}, and it does not depend on the fact that $X$ is affine. 
  %Fix integers $n\leq m$.

Let $\At$ be a termwise flat DG algebra over $R$. We will construct a connection on the filtered DG algebra $\overline{T}(C_p,\At^{\otimes ^p})$
(and, in particular,  on $T_{[n,0]}(C_p,\At^{\otimes p})$, ($n\leq 0$)), that is, a quasi-isomorphism of filtered DG algebras  $$\nabla :p_1^*\overline{T}(C_p,\At^{\otimes p})\cong p_2^*\overline{T}(C_p,\At^{\otimes p})$$ which is, when restricted to $\Delta$ is equal to identity in the derived category of complexes. The exact sequence of sheaves on $X\times X$ $$0\to\Omega^1_{\Delta}\to\cO_{X\times X}/I_{\Delta}^2\to \cO_{\Delta}\to 0$$ induces two exact sequences of complexes 

$$0\to \At\otimes\Omega^1_X\xrightarrow{\beta} p_{1*}p_2^*\At\xrightarrow{\alpha} \At\to 0
$$
$$
0\to \overline{T}(C_p,\At^{\otimes p})\otimes\Omega^1_X\xrightarrow{\beta'} p_{1*}p_2^*\overline{T}(C_p,\At^{\otimes p})\xrightarrow{\alpha'} \overline{T}(C_p,\At^{\otimes p})\to 0 
$$
\def\ja{(p_{1*}p_2^*\At)}
Giving a connection on $\At$ is equivalent to providing a splitting of the latter extension in the category of filtered DG algebras localized with respect to filtered quasi-isomorphisms. We are going to construct such a splitting. Denote by 
$$0\subset G^{p} \ja^{\otimes p} \subset G^{p-1} \ja^{\otimes p}\subset \cdots \subset  G^0 \ja^{\otimes p}= \ja^{\otimes p}$$
 the filtration induced by  $\At\otimes\Omega^1_X\subset  p_{1*}p_2^*\At$. 

\begin{lm}\label{quotients} For a term-wise flat DG algebra $\At$ the morphism $\alpha$ induces the following isomorphism of DG algebras with the action of $C_p$ $$G^0\ja^{\otimes p}/G^1{\ja}^{\otimes p}\xrightarrow{\sim} \At^{\otimes p}$$ and $\beta$ induces the following isomorphism of complexes with the action of $C_p$ $$\At^{\otimes p}\otimes_{\cO_X}\Omega^1_X\otimes_{\bZ}\mathbb{Z}[C_p]\iso G^1{\ja}^{\otimes p}/G^2{\ja}^{\otimes p}$$
\end{lm}

The proof is straightforward.

%\begin{proof} It is enough to prove the statements locally on $X$ so we may choose a connection on the graded module $\bigoplus A_i$ and get a splitting(non-compatible with the differentials) $p_{1*}p^{*}_2\At=\At\oplus %\At\otimes_{\cO_X}\Omega^1_X$, hence \begin{equation}\begin{split}\ja^{\otimes p}=\At^{\otimes p}\oplus\bigoplus\limits_{i=0}^{p-1}(\At\otimes\dots\overset{i}{\otimes(\At\otimes\Omega^1_X)}\otimes\dots\otimes %%\At)\oplus \\
%\bigoplus\limits_{i\neq j}^{p-1}(\At\otimes\dots\overset{i}{\otimes(\At\otimes\Omega^1_X)}\otimes\dots\otimes\overset{j}{(\At\otimes\Omega^1_X)}\otimes\dots\otimes \At)\oplus\dots\end{split}\end{equation} $\alpha$ %projects $\ja^{\otimes p}$ on the first summand of this decomposition while $\beta$ embeds $\At^{\otimes p}\otimes_{\cO_X}\Omega^1_X\otimes_{\bZ}\mathbb{Z}[C_p]$ onto the second (the summand $(\At\otimes\dots%\overset{i}{\otimes(\At\otimes\Omega^1_X)}\otimes\dots\otimes \At)$ corresponds to $\At^{\otimes p}\otimes_{\cO_X}\Omega^1_X\otimes_{\bZ}\bZ[C_p]\otimes\sigma^i$) so they indeed induce isomorphisms on the %graded quotients.
%\end{proof}

By adjunction, we have a map $m: \ja^{\otimes p}\to p_{1*}p_2^*(\At^{\otimes p})$. Since $X\to X^{[2]}$ is a square-zero extension, $m$ factors through $G^2$, so we get the following diagram of complexes of $C_p$-modules in which the top row is a distinguished triangle

\begin{equation}\label{Tatecondef}
\begin{tikzcd}
 \At^{\otimes p}\otimes \Omega^1_X\otimes \bZ[C_p]\arrow[r,"i"] & G^0/G^2\arrow[r,"\pi"]\arrow[d,"m"] & \At^{\otimes p}\arrow[r]&{}\\
 & p_{1*}p_2^*(\At^{\otimes p})
\end{tikzcd}
\end{equation}

The complex of $C_p$-modules $\At^{\otimes p}\otimes \Omega^1_X\otimes \bZ[C_p]$ is isomorphic to the tensor product of the complex $\At^{\otimes p}\otimes \Omega^1_X$ with {\it trivial} $C_p$-action and a free module $\bZ[C_p]$. Thus, Tate cohomology complex of this complex is quasi-isomorphic to $\At^{\otimes p}\otimes \Omega^1_X\otimes T(C_p,\bZ[C_p])$, which is quasi-isomorphic to zero.

Thus, applying the functor $T(C_p,-)$ to (\ref{Tatecondef}),  we  get the following diagram of filtered DG algebras, where $\pi$ is a filtered quasi-isomorphism.

\begin{equation}
\begin{tikzcd}
 T(C_p,G^0/G^2)\arrow[r,"\pi"]\arrow[d,"m"] & T(C_p,\At^{\otimes p})\arrow[r]&{}\\
  p_{1*}p_2^*T(C_p,\At^{\otimes p})
\end{tikzcd}
\end{equation}
Here $T(C_p,\At^{\otimes p})$ is equipped with  the filtration $\tau^{dec}_{\leq \idot}T(C_p,\At^{\otimes p})$, the complex $p_{1*}p_2^*T(C_p,\At^{\otimes p})$ is considered with the  filtration $p_{1*}p_2^*\tau^{dec}_{\leq \idot}T(C_p,\At^{\otimes p})$ and $T(C_p,G^0/G^2)$ is endowed with the image of the filtration $\tau^{dec}_{\leq\idot}T(C_p,G^0)=\tau^{dec}_{\leq\idot}T(C_p,\ja^{\otimes p})$ under the projection $G^0\to G^0/G^2$.

It follows that $\pi$ induces a quasi-isomorphism of the completed Tate cohomology complexes $\pi:  \overline{T}(C_p,G^0/G^2)\iso \overline{T}(C_p,\At^{\otimes p})$. Finally, put $s=m \pi^{-1}$. This is a section of $\alpha'$. By  adjunction, $s$ induces a connection 
\begin{equation}\nabla:p_1^*\overline{T}(C_p,\At^{\otimes p})\cong p_2^*\overline{T}(C_p,\At^{\otimes p})\end{equation}

%Since $\pi:T(C_p,G^0/G^2)\to T(C_p,\At^{\otimes p})$ is a quasi-isomorphism of {\it filtered DG algebras}, the resulting connection is compatible with filtrations and with the DG %algebra structure in the case $m=0$.

\begin{rem}
Using Remark \ref{infinite} and the above argument, one can show that   $T(C_p,\At^{\otimes p})$ can also be endowed with a connection. We will not need this fact. 
\end{rem}

\subsection{The connection on the truncated Tate complex}

%In this section we prove that  the truncated Tate complex 
%$$\cB(\At) = \tau_{[0,1]}T(C_p,\At^{\otimes p}) $$
%regarded as a DG algebra with connection is quasi-isomorphic to
%$$(Fr^*\At\oplus Fr^*\At[1], \nabla_{can} + \mu C^{-1}_{\tF}(\ks) ).$$

\newcommand{\hic}{{\hat i^{cris}}}
\newcommand{\hicd}{{\hat i_{cris*}}}
\newcommand{\hicp}{{\hat i^{cris*}}}
\newcommand{\hR}{{\hat R}}
\newcommand{\hA}{\hat{A}}
\newcommand{\hi}{\hat{i}}
\newcommand{\hFr}{\hat{Fr}}

As in the introduction, denote by $\hicd$ and $\hicp$ respectively the direct and inverse image functors between the categories of crystals on $R$ over $k$ and over $W(k)$. By the virtue of Theorem 6.6 from \cite{bo} we view the category $\mathrm{Cris}(R/k)$ of crystals as a full subcategory of the category of $R$-modules with connection. A key result of this section is the following theorem.

\begin{Th}\label{tate}There is a quasi-isomorphism of DG algebras with connection
\begin{equation}
\cB(\At):= T_{[-1,0]}(C_p,\At^{\otimes p})\cong L\hicp\hicd F^*\At=:\mathcal{B}^{cris}(\At)
\end{equation}
\end{Th}
A proof of Theorem \ref{tate} is given in \S \ref{pftate}.

Now let $\tR$ be the lifting of $R$ over $W_2(k)$ and choose a lifting $\tF$ of the Frobenius morphism on $\tR$. Choose also a lifting $\hR$ of $\tR$ over $W(k)$. Consider the functors 

\begin{equation}
\begin{matrix}
\hi_*:D(Mod-R)\to D(Mod-\hR) &  L\hi^*:D(Mod-\hR)\to D(Mod-R) \\
\ti_*:D(Mod-R)\to D(Mod-\tR) & L\ti^*:D(Mod-\tR)\to D(Mod-R)
\end{matrix}
\end{equation}

Again, by Theorem 6.6 in \cite{bo}, the categories of crystals on $R$ over $W_2(k)$ and $W(k)$ are equivalent to the categories of respectively $\tR$- and $p$-adically complete $\hR$-modules with flat quasi-nilpotent connection. Also, direct and inverse images are compatible with the forgetful functors $u_{R/W_2(k)*}:\mathrm{Cris}(R/W_2(k))\to Mod-\tR;u_{R/W(k)*}:\mathrm{Cris}(R/W(k))\to Mod-\hR$. So, as a complex of $R$-modules $\cB^{cris}(\At)$ is quasi-isomorphic to $L\hi^*\hi_*F^*\At$ (note, however, that the quasi-isomorphism in Theorem \ref{tate} does not depend on any choices of liftings)

\begin{Th}\label{cris}A lifting of $\At$ to a DG algebra $\tilde{\At}/\tR$ gives a quasi-isomorphism of DG algebras with connection
\begin{equation}
\cB^{cris}(\At)\cong \cC^{-1}_{(\tR,\tF)}(\At[\mu], \mu \ks )
\end{equation}
where $\mu$ is a free generator in degree $1$ and $\cC^{-1}$ is the inverse Cartier transform in the sense of \S 1.1.
\end{Th}

\subsection{Proof of Theorem \ref{cris}}

Replace $\tA$ by a semi-free resolution (cf. \cite{dr} 13.4) over $\tR$ and $\At$ by $\tA\otimes_{\tR}R$ (the latter will be automatically semi-free over $\tR$). Fix a connection $\nabla'$ on the free algebra $\bigoplus A_i$. It might not be compatible with the differential -- the Kodaira-Spencer class measures this incompatibility: $\ks=[\nabla',d]$.

\begin{lm}\label{canconlift} For a free module $B/\tilde{R}$ a connection $\nabla_0$ on $\ti^*B$ gives rise to a connection on $\tF^*(B)$ which reduces to the canonical connection on $F^*B$ under $\ti^*$.
\end{lm}

\begin{proof}Lift $\nabla_0$ to a map of $W_2(k)$-modules $\nabla'_0:B\to B\otimes \Omega^1_{\tilde{R}/W_2(k)}$. Then define a connection $\tn$ on $B$ as the pullback of $\nabla'_0$ under $\tF$. Namely, for $f\otimes x\in \tilde{R}\otimes_{\tF,{\tilde{R}}}B$ put
\begin{equation}\tn(f\otimes x)=x\otimes df+f\cdot \tF^*(\nabla'_0(x))
\end{equation}

Since $\nabla'_0$ modulo $p$ is a connection, $\tn$ is actually a well-defined connection on $B$(i.e., does not depend on the way of representing an element of $\tF^*B$ as $f\otimes x$). It does not depend on the choice of $\nabla'_0$ because for a $1$-form $\omega\in \Omega^1_{\tR/W_2(k)}$ the value of $\tF^*(\omega)$ depends only on $\ti^*\omega$ since $\ti^*\tF$ is zero on $1$-forms.
\end{proof}

Applying the lemma to the underlying $\tR$-module $B=\bigoplus\tA_i$ of the given lifting and the connection $\nabla'$, we get a connection $\tn$. Since $\tn$ and $\td$ commute modulo $p$, we get the following $\tR$-linear map

\begin{equation}
\frac{[\tn,\td]}{p}:\tF^*\tA_i\to \ti_*F^*A_{i+1}\otimes \Omega^1_{\tR/W_2(k)}
\end{equation}

We are now ready to prove the theorem. Put $\cF_{\idot}=cone(\ti_*F^*\At\xrightarrow{p}\tF^*\tilde{\At})$. This is a complex of $\tR$-modules with terms $$\cF_i=\tF^*\tA_i\oplus \ti_*F^*A_{i+1}$$ and the differential given by $(x,y)\mapsto (d_{\tA}x+(-1)^ipy,d_{A}y)$. 

Let $r:\cF_{\idot}\to \ti_*F^*\At$ be the morphism which maps $(x,y)\in \cF_i$ to the reduction of $x$ modulo $p$ in $\ti_*F^*A_i$. Note that $r$ is a morphism of complexes because $p\in \tR$ acts by zero on $\ti_*F^*\At$. 

\newcommand{\ih}{\hat{i}}

\begin{lm}\label{niceres}(i) $r$ is a quasi-isomorphism.

(ii) Considering further $\cF_{\idot}$ as a complex of $\hR$-modules, the canonical map $L \hi^*\cF_{\idot}\to \hi^*\cF_{\idot}$ is a quasi-isomorphism.
\end{lm}

\begin{proof}
(i) It is clear as $r$ is term-wise surjective and its kernel is isomorphic to $cone(\ti_*F^*\At\xrightarrow{id}\ti_*F^*\At)$ which has zero cohomology.

(ii) Terms of $\cF_{\idot}$ are not flat over $\hR$ so, a priori, there might be non-zero higher derived functors of $\hi^*$. Let $\bigoplus \cA_i$ be a lifting of the graded algebra $\bigoplus \tF^*\tA_i$ to a free graded algebra over $\hR$. Pick also a lifting $\cd:\bigoplus \cA_i\to \bigoplus \cA_i[1]$ of the differential $\td$ ($\cd$ is not a differential anymore -- its square need not be zero). It enables us to write down the following resolution of $\ih_*\At$. Put 

\begin{equation}\label{flatres}C_i=\cA_i\oplus \cA_{i+1}; d_C=\left(\begin{matrix}\cd & (-1)^ip\\ (-1)^i\frac{\cd^2}{p} & \cd\end{matrix}\right)
\end{equation}

$\cd^2$ is divisible by $p$ because $d^2=0$ on $F^*\At$ and modules $\cA_i$ are free over $\hR$. Moreover, $\delta^2$ is divisible by $p^2$ because $\td^2=0$ on $\tA$. Hence, $\frac{\delta^2}{p}$ is divisible by $p$, so the reduction maps $C_i\to \cF_i$ give a morphism of complexes $\rho:C_{\idot}\to\cF$. Actually, $\rho$ is a quasi-isomorphism. Indeed, composing it with $r$ we get a term-wise surjective morphism of complexes with kernel given by $K_i=p\cA_i\oplus \cA_{i+1}$ and the differential restricted from $C_{\idot}$. For any $(x,y)\in K_i$ such that $d_C(x,y)=0$ we have $(x,y)=d_C(0,(-1)^{i-1}\frac{x}{p})$ so $K_{\idot}$ is acyclic and $C_{\idot}$ is an $\hR$-flat resolution of $\cF$.  We get a commutative diagram

\[
\begin{tikzcd}
L\ih^*C_{\idot}\arrow[r,"\sim"]\arrow[d] & \ih^*C_{\idot}\arrow[d] \\
L\ih^*\cF_{\idot}\arrow[r] & \ih^*\cF_{\idot}
\end{tikzcd}
\]

Left vertical arrow is a quasi-isomorphism because $C_{\idot}\to\cF_{\idot}$ is a quasi-isomorphism and the right vertical arrow is an isomorphism because both $C_{\idot},\cF_{\idot}$ reduce modulo $p$ to the complex $F^*\At\oplus F^*\At[1]$. Thus, the lower arrow is a quasi-isomorphism.
\end{proof}

Endow  $\cF$ with the structure of a DG algebra with connection. As a DG algebra $\cF$ is the trivial square-zero extension of $\tF^*\tilde{\At}$ by the bimodule $\ti_*F^*\At[1]$. Explicitly, the product of $(x,y)\in \tF^*\tA_i\oplus \ti_*F^*A_{i+1}$ and $(x',y')\in \tF^*\tA_j\oplus \ti_*F^*A_{j+1}$ is defined to be $(xx',(-1)^jyx'+(-1)^ixy')$. To see that this algebra structure is compatible with the differential it is enough to check that $D:(x,y)\mapsto ((-1)^ipy,0)$ is a derivation because the diagonal part $(x,y)\mapsto (d_{\tA}x,d_{A}y)$ of $d_{\cF}$ is a derivation by default. For $(x,y)\in \cF_i,(x',y')\in \cF_j$ we have $D((x,y)(x',y'))=((-1)^{i+j}p((-1)^jyx'+(-1)^ixy'),0)=((-1)^ipy,0)(x',y')+(x,y)((-1)^jpy',0)=D((x,y))(x',y')+(x,y)D((x',y'))$.

Next, define a connection by

\begin{equation}
\nabla_{\cF}=\left(
\begin{matrix}
\tn & 0 \\ (-1)^i\frac{[\tn,\td]}{p} & \ti_*\nabla^{can}
\end{matrix}
\right):\tF^*\tilde{A}_i\oplus \ti_*F^*A_{i+1}\to (\tF^*\tilde{A}_{i}\oplus \ti_*F^*A_{i+1})\otimes_{\tR} \Omega^1_{\tR/W_2(k)}
\end{equation}

The entry below the diagonal is chosen so that this connection commutes with the differential on the DG algebra. To ensure that this connection respects the algebra structure  it is, as above, enough to check that $(x,y)\mapsto (0,(-1)^i\frac{[\tn,\td]}{p}x)$ is a derivation which follows from ${[\tn,\td]}$ being a commutator of derivations. Finally, it is clear that our connection is integrable.

Also, quasi-isomorphism $r$ is compatible with connection because $\tn$ reduces to $\nabla^{can}$ modulo $p$. In other words, $\hicd F^*\At$ is quasi-isomorphic to $(\cF,\nabla_{\cF})$. Thus, $T^{cris}(\At)\cong L\hicp((\cF,\nabla_{\cF}))$. By the virtue of Lemma \ref{niceres}, $L\hicp(\cF,\nabla_{\cF})$ is quasi-isomorphic to $(i^*\cF,i^*\nabla_{\cF})$. The latter complex of $R$-modules with integrable connection is given by

\begin{equation}
\left(
\begin{matrix}
\nabla^{can} & 0 \\ (-1)^i\frac{[\tn,\td]}{p} & \nabla^{can}
\end{matrix}
\right):F^*A^i\oplus F^*A_{i+1}\to (F^*A_{i}\oplus F^*A_{i+1})\otimes \Omega^1_{R/k}
\end{equation} 

So, Theorem \ref{cris} will follow after we check that

\begin{lm}\begin{equation}\frac{[\tn,\td]}{p}=C^{-1}_{\tR,\tF}(\ks)
\end{equation}
\end{lm}

\begin{proof}By definition $\ks=[\nabla',d]$.  Recall from the Lemma \ref{canconlift} that $\tilde{\nabla}_i$ on $\tF^*\tilde{A}_i$ is given by the formula $\tn_i(f\otimes x)=df\otimes x+f\otimes \tilde{F}^*(\nabla'_i(x))$. Hence,

\begin{equation}
\begin{split}\frac{[\tn,\td]}{p}(f\otimes x)=\frac{df\otimes \td(\tx)+f\otimes\tilde{F}^*(\nabla'_i(dx))-df\otimes \td(\tx)-f\otimes d\tF^*(\nabla'_i(x))}{p}=\\= f\otimes \frac{\tF^*([\nabla'_i,d])}{p}
\end{split}
\end{equation} which is exactly the inverse Cartier operator of the Kodaira-Spencer class by the definition (\ref{inversecartieroperator}).

\end{proof}

\begin{rem}
Of course, we could have computed $\cB^{cris}(\At)$ in one step using the resolution (\ref{flatres}) but we deal with non-liftability of $\At$ over $W(k)$ and non-existence of a connection on $\At$ separately for the sake of exposition. 
\end{rem}

\subsection{Proof of Theorem \ref{tate}}\label{pftate}

Choose a lifting $\hat{F}:\hR\to \hR$ of the Frobenius endomorphism such that $\hat{F}\otimes_{\hR}\tR=\tF$.

\begin{lm}\label{h1(W)}
Let $M$ be a flat $\hR$-module. For any $n\in \mathbb{Z}$ we have $\hat{H}^{2n-1}(C_p, M^{\otimes p})=0$ and $\hat{H}^{2n}(C_p,M^{\otimes p})$ is canonically isomorphic to $\ih_*F^*\ih^*M$,  where $C_p$, as usual, acts on $M^{\otimes p}$ by cyclic permutations.
\end{lm}

\begin{proof}The proof is similar to that of the Lemma 6.9 in \cite{cop}. By periodicity, it is enough to consider the case $n=0$. So, we should compute cohomology of the following canonical truncation of the Tate complex $$(M^{\otimes p})_{C_p}\xrightarrow{N} (M^{\otimes p})^{C_p}$$ Lemma 6.9 from \cite{cop} gives for any flat $R$-module $N$ a map $(N^{\otimes p})^{C_p}\to F^*N$. Composing this map for $N=\ih^*M$ with the inclusion $\ih^*(M^{\otimes p})^{C_p}\to (\ih^*M^{\otimes p})^{C_p}$ we get a map $\psi: \ih^*(M^{\otimes p})^{C_p}\to F^*\ih^*M$ which, by adjointness, gives a map of complexes $$[(M^{\otimes p})_{C_p}\to (M^{\otimes p})^{C_p}]\to \ih_*F^*\ih^*M$$ Since any flat module is a filtered colimit of free modules, it is enough to prove that this map is a quasi-isomorphism for finitely-generated free modules. Fixing a basis $S$ in a free module $M$, we get a decomposition of $C_p$-modules $$M^{\otimes p}=M_1\oplus M_2$$ where $M_1$ is generated by $s^{\otimes p}$ for $s\in S$ and $M_2$ is generated by all other tensors. So, $M_1$ is a trivial $C_p$-module, while $M_2$ is free and $\psi$ factors through projection on $M_1$.  So, to prove the lemma it is left to check that $\hat{H}^{-1}(C_p,\hR)=0,\hat{H}^{0}(C_p,\hR)=\ih_*R$. The standard Tate complex for trivial module $\hR$ takes the following form \begin{equation}\dots\xrightarrow{0}\hR\xrightarrow{p}\hR\xrightarrow{0}\dots\end{equation} So, $\hat{H}^{0}(C_p,\hR)=\hR/p\hR=\ih_*R,\hat{H}^{-1}(C_p, \hR)=0$ because multiplication by $p$ is injective on $\hR$.
\end{proof}

In what follows, for any DG algebra $\Bt$ we write $T(\Bt)$ for the DG algebra $T(C_p,\Bt^{\otimes p})$.

\begin{pr}\label{splcase}
A lifting $\hA_{\idot}$ of $\At$ over $\hR$ gives rise to a quasi-isomorphism of DG algebras
\begin{equation}
T_{[-1,0]}(\At)\cong L\ih^*\ih_*F^*\At
\end{equation}
\end{pr}

\begin{proof}
By definition, $T_{[-1,0]}(\At)\cong L\ih^*T_{[-1,0]}(\hA_{\idot})$. Replacing in the proof of Proposition 6.10 from \cite{cop} their Lemma 6.9 by our \ref{h1(W)} we get that $T_{[-1,-1]}(\hA_{\idot})=0,T_{[0,0]}(\hA_{\idot})=\ih_*F^*\ih^*\hA_{\idot}=\ih_*F^*\At$. The vanishing of $T_{[-1,-1]}$ implies that $T_{[-1,0]}T(\hA_{\idot})\to T_{[0,0]}T(\hA_{\idot})\cong \ih_*F^*\At$ is an isomoprhism. Applying $L\ih^*$ we get the statement.

\end{proof}

\newcommand{\hF}{\hat{F}}

For a liftable $\At$, the above proposition can be reformulated as $$T_{[-1,0]}(\At)\cong F^*\At\oplus F^*\At[1]$$
 because $\hF^*\hA_{\idot}$ is a lifting of $F^*\At$ which splits $L\ih^*\ih_*F^*\At$ by Theorem \ref{cris}.

Next, if $\At$ is arbitrary, apply the proposition to $L\ih^*\ih_*\At$ letting $\hA_{\idot}$ be a semi-free resolution of $\ih_*\At$. We get \begin{equation}\label{dec}T_{[-1,0]}(L\ih^*\ih_*\At)\cong L\ih^*\ih_*F^*\At\oplus L\ih^*\ih_*F^*\At[1]\end{equation} Consider the morphism $L\ih^*\ih_*\At\to \ih^*\ih_*\At=\At$. It induces the map $T(L\ih^*\ih_*\At)\to T(\At)$. So we get the following diagram 
$$
\begin{tikzcd}
L\ih^*\ih_*F^*\At\arrow[r] & L\ih^*\ih_*F^*\At\oplus L\ih^*\ih_*F^*\At[1]\cong T_{[-1,0]}(L\ih^*\ih_*\At) \arrow[r] & T_{[-1,0]}(\At)
\end{tikzcd}
$$

Denote the composition by $\varphi$. Note that, by construction, $\varphi$ is a morphism of DG algebras. First,

\begin{lm}\label{qis}$\varphi:L\ih^*\ih_*F^*\At\to T_{[-1,0]}(\At)$ is a quasi-isomorphism of complexes of $R$-modules.
\end{lm} 

\begin{proof}

We may use the resolutions from the proof of Lemma \ref{niceres} to compute $L\ih^*\ih_*F^*\At$. Since functors $\ih^*,\ih_*,F^*,\At\mapsto \cF$ commute with filtered colimits, we may assume that $\At$ is a perfect complex (any complex is a filtered colimit of perfect complexes). Next, it is enough to check that  $\varphi $ is a quasi-isomorphism over all the localizations $R_{\mathfrak{m}}$ at maximal ideals $\mathfrak{m}\subset R$. Finally, by Nakayama lemma, it is enough to verify the statement over residue fields $R/\mathfrak{m}$.

Note that for any $\At$, the following square is commutative 

$$
\begin{tikzcd}
L\ih^*\ih_*F^*\At\arrow[r, "\varphi"]\arrow[d] & T_{[-1,0]}(\At)\arrow[d]\\
F^*\At\arrow[r, equal] & F^*\At
\end{tikzcd}
$$

Since any vector space is a filtered colimit of finite dimensional ones and finite dimensional vector spaces are finite direct sums of one-dimensional space, it is enough to prove the statement for $k'=R/\mathfrak{m}$. $L\ih^*\ih_*F^*k'$ and $T_{[-1,0]}(k')$ are both non-canonically split, i.e. quasi-isomorphic to $k'\oplus k'[1]$ and $\varphi$ induces an isomorphism on zeroth cohomology. We should prove that it is also an isomorphism on $(-1)$-st cohomology. Assume it is not, i.e. is zero on $H^{-1}$. Then $\varphi$ factors through $L\ih^*\ih_*F^*k'\to F^*k'$ so induces a splitting of $T_{[-1,0]}(k')$. Since, $\varphi$ is compatible with direct sums, $T_{[-1,0]}(V)$ is also canonically split for any $k'$-vector space $V$. In other words, the following extension of polynomial functors $Vect_{k'}\to Vect_{k'}$ is split $$0\to F^*V\to (V^{\otimes p})_{C_p}\to (V^{\otimes p})^{C_p}\to F^*V\to 0$$

This extension is equivalent to a similar one with $C_p$ replaced by the symmetric group $S_p$

$$
\begin{tikzcd}
0\arrow[r] & F^*V\arrow[r]\arrow[d, equal] & (V^{\otimes p})_{C_p}\arrow[d, "\pi_p"]\arrow[r,"N_{C_p}"] & (V^{\otimes p})^{C_p}\arrow[d, "av_p"]\arrow[r] & F^*V\arrow[d, equal]\arrow[r] & 0\\
0\arrow[r] & F^*V\arrow[r] & (V^{\otimes p})_{S_p}\arrow[r, "N_{S_p}"] & (V^{\otimes p})^{S_p}\arrow[r] & F^*V\arrow[r] & 0
\end{tikzcd}
$$

Here $\pi_p$ is the projection and $av_p$ is the averaging over left cosets of $C_p\subset S_p$ that is $av_p(x)=\frac{1}{(p-1)!}\sum\limits_{gC_p\in S_p/C_p}g(x)$(note that this does not depend on the choice of representatives of cosets). From Corollary 4.7($r=j=1$) and Lemma 4.12 from \cite{fs} follows that the latter extension is non-split. Hence, $\varphi$ must induce an isomorphism on (-1)-st cohomology so it is a quasi-isomorphism for any $\At$.
\end{proof}

We have constructed a map $\varphi:\cB^{cris}(A)\to\cB(\At)$ of DG algebras over $R$. To finish the proof of the theorem we need to prove that

\begin{lm} $\varphi$ is compatible with the connection in the sense that it is a morphism in the category of complexes with connection localized with respect to quasi-isomorphisms.
\end{lm}

\begin{proof}
First, assume that lemma is proven for liftable DG algebras, in particular for $L\ih^*\ih_*\At$. Theorem \ref{cris} implies that embedding $L\ih^*\ih_*F^*\At\to L\ih^*\ih_*F^*L\ih^*\ih_*\At$ is compatible with connection because the Kodaira-Spencer class of $F^*\At$ vanishes. The morphism $T_{[-1,0]}(L\ih^*\ih_*\At)\to T_{[-1,0]}(\At)$ is also compatible with the connection because, by definition, connection on the Tate complex is functorial in the DG algebra. So, $\varphi$ is a composition of morphisms compatible with the connection.
 
So, we may assume that $\At$ has a lifting $\hA_{\idot}$ over $\hR$. We claim that $T_{[-1,0]}(\hA_{\idot})\to \ih_*F^*\At$ is compatible with connections where the truncated Tate complex carries the connection from \ref{ConT} and $\ih_*F^*\At$ is the direct image of the canonical connection. The map $T_{[-1,0]}(\hA_{\idot})\to T_{[0,0]}(\hA_{\idot})$ is obviously compatible, so we need to check that the isomorphism $T_{[0,0]}(\hA_{\idot})\cong \ih_*F^*\At$ is compatible. Applying $T_{[0,0]}$ to the diagram (\ref{Tatecondef}) used in the definition of connection, we get

\[
\begin{tikzcd}
 T_{[0,0]}(C_p,G^0/G^2(\ih_*p_{1*}p^*_2\At)^{\otimes p})\arrow[r,"\pi"]\arrow[d,"m"] & T_{[0,0]}(C_p,(\ih_*\At)^{\otimes p})\\
 T_{[0,0]}(C_p,p_{1*}p_2^*(\ih_*\At)^{\otimes p})
\end{tikzcd}
\]

By the proof of \ref{splcase}, $T_{[0,0]}(C_p,(\ih_*\At)^{\otimes p})=\ih_*F^*\At$ and, similarly, $\pi$ induces an isomorphism, because the kernel of $\pi:G^0/G^2((\ih_*p_{1*}p^*_2\At)^{\otimes p})\to (\ih_*\At)^{\otimes p}$ is a free complex of $C_p$-modules, thus has Tate cohomology complex quasi-isomorphic to zero. Finally, since $p_{1*}p^{*}_2$ commutes with $T_{[0,0]}$, we get 

\[
\begin{tikzcd}
 \ih_*F^*\At\arrow[r,"\Id"]\arrow[d,"m"] & \ih_*F^*\At\\
 p_{1*}p^*_2 \ih_*F^*\At
\end{tikzcd}
\]

So, indeed, $T_{[0,0]}(\hA)$ is isomorphic to the $i_*$ of the canonical connection on $F^*\At$.

\end{proof}

\section{The Gauss-Manin connection on the (co-)periodic cyclic homology}\label{sGMonHP}
In this section we review Getzler's and Kaledin's constructions of the Gauss-Manin connection, check that the two constructions agree, show that the Gauss-Manin connection preserves the conjugate filtration,  and prove   Theorem \ref{cBandCP}. 
\subsection{Getzler's construction.} 
Let $R$ be a smooth finitely generated commutative  algebra over a field $k$, and let $\At$ be  a semi-free differential graded algebra  over $R$ (\cite{dr}, \S 13.4). Denote by $(\Ct(\At, \At), b)$ the relative Hochschild chain complex of $\At$  over $R$
\footnote{Here  ``relative over $R$'' means that all the tensor products in the standard complex are taken over $R$.} and by $\CPt (\At)= (\Ct (\At, \At)((u)), b +uB)$ the periodic cyclic complex. Getzler defined in \cite{ge} a connection on $\CPt (\At)$ 
$$\nabla:  \CPt (\At) \to \CPt (\At) \otimes _R \Omega^1_R.$$
His construction can be explained as follows: choose
a connection $\nabla ':  \bigoplus A_i \to \bigoplus A_i \otimes \Omega^1_R$ on the graded algebra $\bigoplus A_i$ satisfying the Leibnitz rule with respect to the multiplication on $\bigoplus A_i$. Then the commutator 
$$\ks= [\nabla', d] \in \prod Hom_R (A_i, A_{i+1}) \otimes \Omega^1_R $$
with the differential $d$ on $\At$ 
commutes with $d$ and it is a $R$-linear derivation of $\At$  (with values in  $\At  \otimes \Omega^1_R $)  of degree $1$. \footnote{Denote by $Der_R^{\bullet}(\At)$ the DG Lie algebra of $R$-linear derivations 
of $\At$: $Der_R^{i}(\At)$ is the $R$-module of $R$-linear derivations of the graded algebra $\bigoplus A_i$; the differential on $Der_R^{\bullet}(\At)$ is given by the commutator with $d$. 
The cohomology class 
$\kappa \in H^1(Der_R^{\bullet}(\At))\otimes \Omega^1_R$ of $\ks$ does not depend on the choice of $\nabla'$.   (Indeed, any two connections differ by an element of $Der_R^{0}(\At)$.)  Recall that the Hochschild cochain complex of $\At$ is quasi-isomorphic to the cone of the map $\At \to Der_R^{\bullet}(\At)$ which takes an element of $A_i$ to the corresponding inner derivation. We refer to the image 
$\overline{\kappa}$ of $\kappa$ under the induced morphism  $H^1(Der_R^{\bullet}(\At))\otimes \Omega^1_R \to HH^2(\At, \At)\otimes \Omega^1_R$ as the Kodaira-Spencer class of $\At$. In \cite{ge}, Getzler chooses local coordinates $x_1, \cdots x_n$ on $\text {spec} R$. His notation ({\it loc. cit.}, \S3) for $\overline{\kappa}$ coupled with $\frac{d}{dx_i}$ is $\cA_i$.}
As a derivation, $\ks$   acts on  $\Ct(\At, \At)$ by the Lie derivative
$$\cL_{\ks}: \Ct(\At, \At) \to \Ct(\At, \At)  \otimes \Omega^1_R [1], \quad [\cL_{\ks}, B]=0$$ 
and the ``interior product'' operator 
$$e_{\ks}: \Ct(\At, \At) \to \Ct(\At, \At)  \otimes \Omega^1_R [2].$$
 The operators $\cL_{\ks}, e_{\ks}, B$ satisfy the Cartan formula up to homotopy:  there is a canonical operator 
 $$E_{\ks}: \Ct(\At, \At) \to \Ct(\At, \At)  \otimes \Omega^1_R, \quad [E_{\ks}, B]=0$$
 such that $[e_{\ks}, B]= \cL_{\ks} - [E_{\ks}, b]$ ( \cite{l}, \S 4.1.8).
One defines
\begin{equation}\label{GCon} \nabla : = \nabla' - u^{-1} \iota_{\ks},\end{equation}
where the first summand  is the connection on $\bigoplus CP_i (\At)$ induced the connection $\nabla '$ on $ \bigoplus A_i $ and 
 $\iota_{\ks}: \bigoplus CP_i (\At) \to \bigoplus CP_i (\At) \otimes \Omega^1_R$ is an $R((u))$ linear map given by the formula  $\iota_{\ks}=e_{\ks} + u E_{\ks}$. By construction, $ \nabla$ commutes with $b +uB$. Thus, it induces a connection on $\CPt (\At)$. 
 Getzler  showed that up to homotopy $\nabla$ does not depend on the choice of $\nabla '$.\footnote{One can rephrase the above construction to make this fact obvious: let $Der_k^{\bullet}(R\to \At)$ be the DG Lie algebra of $k$-linear derivations which take the subalgebra   $R \subset A_0$ to itself. Then
 $Der_R^{\bullet}(\At)$ is a Lie ideal in  $Der_k^{\bullet}(R\to \At)$. Denote by $\widetilde{Der_k(R)}$ the cone of the morphism $Der_R^{\bullet}(\At)\to Der_k^{\bullet}(R\to \At)$. The restriction morphism
  $\widetilde{Der_k(R)}\to Der_k(R)$  a homotopy equivalence of DG Lie algebras: a choice of $\nabla'$ as above yields a homotopy inverse map. 
 Next, we have
 a canonical morphism of complexes $\widetilde{Der_k(R)} \otimes _R  \CPt (\At) \to \CPt (\At) $ given by the formulas $\theta \otimes c \mapsto  u^{-1} \iota_{\theta} (c)$,  for $\theta \in  Der_R^{\bullet}(\At)$, and  
 $\zeta \otimes c \mapsto  \cL_{\zeta}( c)$, for  $\zeta \in  Der_R^{\bullet}(R\to \At)$. This yields a morphism  $ Der_k(R) \otimes _R  \CPt (\At) \to \CPt (\At) $  well defined up to homotopy.} 
 He also proved that the induced connection on $HP_{\bullet}(\At)$ is flat. However, we do not know how to make $\nabla$ on  $ \CPt (\At)$ flat up to coherent homotopies\footnote{The problem is that, in general, the canonical morphism $\widetilde{Der_k(R)} \otimes _R  \CPt (\At) \to \CPt (\At) $ is not a Lie algebra action.}.

By construction, the connection $\nabla$  satisfies the Griffiths transversality property with respect to the Hodge filtration 
 $\cF^i\CPt (\At):= (u^i \Ct (\At, \At)[[u]], b +uB)$:
  $$\nabla:  \cF^i\CPt (\At) \to \cF^{i-1}\CPt (\At) \otimes _R \Omega^1_R.$$
 Thus, $\nabla$ induces a degree one $R$-linear morphism of graded complexes 
 $$Gr^\cF {\nabla}:  Gr^\cF \CPt \to Gr^\cF \CPt \otimes _R \Omega^1_R.$$
 Abusing terminology,  we refer to  $Gr^\cF {\nabla}$ as the Kodaira-Spencer operator.  Under the identification $Gr^\cF \CPt= (\Ct (\At, \At)((u)), b)$ the Kodaira-Spencer operator is given by the formula
 $$Gr^\cF {\nabla}= u^{-1} e_{\ks}.$$
\newcommand{\jta}{p_{1*}p^*_2\At}

\subsection{Kaledin's definition.} Following (\cite{ka1}, \S 3), we extend the argument from \S \ref{ConT} to give another definition of the Gauss-Manin connection which will be used in our proofs. Consider a two-term filtration on $p_{1*}p^*_2\At$ given as $I^0=\jta, I^1=\At\otimes\Omega^1_X,I^2=0$. Note that $I^0/I^1=\At$. Taking  tensor powers of the filtered complex $p_{1*}p^*_2\At$, we obtain a filtration on the cyclic object 
$\ja^{\#}$. This gives rise to a filtration $I^i$ on the periodic cyclic complex of $\jta$ such that $I^0\CPt\ja/I^1\CPt\ja=\CPt(\At)$. So we get a diagram with the upper row being a distinguished triangle 

\[\label{CPFil}
\begin{tikzcd}
 I^1/I^2\arrow[r,"i"] & I^0/I^2\arrow[r,"\pi"]\arrow[d,"m"] & \CPt(\At)\arrow[r]&{}\\
 & p_{1*}p_2^*\CPt(\At)
\end{tikzcd}
\]

\begin{lm}

$I^1\CPt\ja/I^2\CPt\ja$ is contractible.

\end{lm}

\begin{proof} By \cite{ka1} \S 3, the cyclic object $I^1\ja^{\#}/I^2\ja^{\#}$ is freely generated by $A^{\#}\otimes \Omega^1$ so its periodic cyclic complex is contractible. 
\end{proof}

Hence, $\pi$ is a quasi-isomorphism, and the connection is defined as $\nabla=m\pi^{-1}:\CPt(\At)\to p_{1*}p^*_2\CPt(\At)$

\begin{pr}\label{GeqK}Kaledin's connection is equal to Getzler's connection as a morphism $\CPt(\At)  \to p_{1*}p^*_2\CPt(\At)$ in the derived category.
\end{pr}

\begin{proof}

We will show that Getzler's formula comes from a section of $\pi$ on the level of complexes. 

%\begin{lm}Let $A\xrightarrow{i} B\xrightarrow{\pi} C\to $ be a distinguished triangle, $\varphi:C_i\to B_i$ -- a section of $\pi$ non compatible with the differential. Then $[\varphi,d]:C\to B[1]$ factors through $A[1]$ and is compatible with the differential. If $[\varphi,d]:C\to A[1]$ is homotopic to $0$ via $h:C\to A$ then $\varphi-ih$ is a section of $\pi$ compatible with the differential

%\end{lm}

$\nabla'$ gives rise to a section $\varphi$ of $\pi:\bigoplus CP_i\ja\to \bigoplus CP_i(\At)$ because $\nabla'$ yields a connection on any $A^{i_1}\otimes\dots\otimes A^{i_k}$ by the Leibnitz rule. Note that

\begin{equation}
\begin{split}
&[\varphi, b](a_0\otimes\dots\otimes a_n)=(\sum 1\otimes\dots\otimes \stackrel{i}{\nabla'}\otimes\dots\otimes 1)(\sum a_0\otimes\dots\otimes da_i\otimes \dots \otimes a_n+\\&+ \sum (-1)^i a_0\otimes\dots\otimes a_ia_{i+1}\otimes\dots\otimes a_n )-b(\sum a_0\otimes\dots\otimes \nabla'a_i\otimes\dots\otimes a_n)=\\&= \sum a_0\otimes\dots\otimes (\nabla'd-d\nabla')a_i\otimes\dots\otimes a_n
\end{split}
\end{equation} 

This computation shows that $[\varphi, b+uB]=\cL_{\ks}$ (because, clearly $[\varphi,B]=0$) where $\cL_{\ks}:\CPt(\At)\to \CPt(I^1\ja^{\#}/I^2\ja^{\#})$. By \cite{l}, \S 4.1.11 we have $[u^{-1}\iota_{\ks},b+uB]=\cL_{\ks}$. Hence, $\varphi-u^{-1}\iota_{\ks}$ is a morphism of complexes and a section of $\pi$ so, in the derived category, $\pi^{-1}=\varphi-u^{-1}\iota_{\ks}$. Applying $m$ we get precisely the map (\ref{GCon}) considered as a map $\CPt(\At)\to p_{1*}p^{*}_2\CPt(\At)$.
 
\end{proof} 
 
\subsection{Proof of Theorem \ref{cBandCP}.}
As explained in  (\cite{ka3}, \S 3.3 and \S 5.1) we have a canonical morphism
\begin{equation}\label{eqkconst0}
\cB(\At)^\hush \to   \pi_{(-2(p-1), 0]}^\flat i_p^*  \At^\hush 
\end{equation}
in $D(\Lambda, R)$. This induces a morphism of cyclic complexes
$$CC_{\idot}(\cB(\At))= CC_{\idot}(\cB(\At)^\hush ) \to CC_{\idot}( \pi_{(-2(p-1), 0]}^\flat i_p^*  \At^\hush ),$$
\begin{equation}\label{eqkconst}
V_{-1}CC_{\idot}(\cB(\At)) \to CC_{\idot}( \pi_{(-2(p-1), -1]}^\flat i_p^*  \At^\hush )\iso V_{[-p+2, -1]} \overline CP_{\idot}(\At)
\end{equation}
We have to check that (\ref{eqkconst}) factors through $V_{[-p+2, -1]}CC_{\idot}(\cB(\At))$ 
and that the resulted morphism is a quasi-isomorphism.

Recall that any complete resolution (in the sense of \S \ref{sstcc}) has the  structure of an algebra over an $E_\infty$ operad (see \cite{Lan}, \S 2, for a construction of this operad).   This makes  $\cB(R)^\hush$ and $\pi_{(-2(p-1), 0]}^\flat i_p^*  R^\hush$  into  $E_\infty$ algebras 
in the category of complexes over $Fun(\Lambda, R)$ and $\cB(\At)^\hush $,    $\pi_{(-2(p-1), 0]}^\flat i_p^*  \At^\hush$ are modules over these algebras respectively. The morphism \ref{eqkconst0}
can be promoted to 
\begin{equation}\label{eqkconst1}
\cB(\At)^\hush  \stackrel{L}{\otimes}_{\cB(R)^\hush} \pi_{(-2(p-1), 0]}^\flat i_p^*  R^\hush \to   \pi_{(-2(p-1), 0]}^\flat i_p^*  \At^\hush 
\end{equation}
Moreover, if we endow the left-hand side of (\ref{eqkconst1}) with the filtration induced by the canonical filtration on $\pi_{(-2(p-1), 0]}^\flat i_p^*  R^\hush $ and the right-hand side with $\tau^{dec}$, then
(\ref{eqkconst1}) is a filtered quasi-isomorphism. Pass to mixed complexes:
\begin{equation}\label{eqkconst2}
C(\cB(\At))  \stackrel{L}{\otimes}_{C(\cB(R))} C(\pi_{(-2(p-1), 0]}^\flat i_p^*  R^\hush ) \to   C(\pi_{(-2(p-1), 0]}^\flat i_p^*  \At^\hush) 
\end{equation}
Now Theorem \ref{cBandCP} follows from an easy Lemma below.
\begin{lm} 
The homomorphism of $E_\infty$ algebras 
$$C(\cB(R)) \to C(\pi_{(-2(p-1), 0]}^\flat i_p^*  R^\hush )$$
induces a quasi-isomorphism
$$\tau_{(-2(p-1), 0]}C(\cB(R)) \iso C(\pi_{(-2(p-1), 0]}^\flat i_p^*  R^\hush ).$$
\end{lm}

\section{The local monodromy theorem}\label{SLMT} 
In this section we prove  Theorem \ref{lmth}  in a stronger and more general form.
We start by recalling some results of Katz from (\cite{katz1}).

\subsection{Katz's Theorem.} Let $S$ be a smooth geometrically connected complete curve over a field $K$ of characteristic $0$, $K(S)$ the field of rational functions on $S$, and let $E$ be a finite-dimensional vector space
over $K(S)$ with  a  $K$-linear connection 
$$\nabla: E \to E \otimes \Omega^1_{K(S)/K}.$$
Recall  that  $\nabla$ is  said to have regular singularities  if $E$ can be extended to a vector bundle $\cE$ over $S$ such that  $\nabla$ extends to a connection on 
 $\cE$, which has  at worst simple poles at some finite closed subset $D\subset S$:
$$\overline \nabla: \cE \to \cE \otimes \Omega^1_{S}(\log D).$$
One says that the local monodromy of $(E, \nabla)$ is quasi-unipotent if the pair  $(\cE, \overline \nabla)$ as above can be chosen so that the residue of $\overline \nabla $
 $$ \Res \overline \nabla:  \cE_{| D}\to  \cE_{| D} $$
 has rational eigenvalues\footnote{One can show (see {\it e.g.}, \cite{katz1}, \S 12) that if $ \Res \overline \nabla$ has  rational eigenvalues for one extension then it has  rational eigenvalues for every extension  $(\cE, \overline \nabla)$ of $(E, \nabla)$.}. Let  $\Res \overline \nabla = D +N$ ,  with $[D,N]=0$,  be the Jordan decomposition of  $\Res \overline \nabla$ as a sum of a semi-simple operator $D$ and a nilpotent operator
 $N$. If the local monodromy of $(E, \nabla)$ is quasi-unipotent we say its exponent of nilpotence is $\leq \nu$ if $N^\nu =0$.

If $K=\bC$ then the category of finite-dimensional $K(S)$-vector spaces with $K$-linear connections with regular singularities and  quasi-unipotent local monodromy is equivalent to
the category of local systems (in the topological sense)  over $S$ take off  finitely many points whose local monodromy around every puncture is  quasi-unipotent
({\it i.e.}, all its eigenvalues are roots of unity).  The exponent of nilpotence of local monodromy is the size of its largest Jordan block.  

In  (\cite{katz1}, Th. 13.0.1), Katz proved the following result.

 \begin{theo}[Katz]
Let $C$ be a smooth scheme of relative dimension $1$ over a domain $R$ which is  finitely generated (as a ring) over $\bZ$,  with fraction field $K$ of characteristic zero. Assume 
that the generic fiber of $C$ is geometrically connected. Let $(M,\nabla)$ be a locally free $\cO_C$-module with a connection $\nabla: M \to M\otimes \Omega^1_{C/R}$.  Assume that $(M,\nabla)$ is globally nilpotent of nilpotence $\nu$, that is, for any prime number $p$, the $\cO_{C\otimes \bF_p}$-module $M\otimes \bF_p$  with $R\otimes \bF_p$-linear connection admits a filtration $$0=V_0(M\otimes \bF_p) \subset \cdots \subset V_\nu (M\otimes \bF_p)=  M\otimes \bF_p$$
 such that the $p$-curvature of each successive quotient $V_i/V_{i-1}$ is $0$. Then the pullback  $M\otimes_{\cO(C)} K(C)$ of $M$ to the generic point of $C$  has regular singularities and quasi-unipotent local monodromy of exponent $\leq \nu$.
\end{theo}
\subsection{Monodromy Theorem.}    Now we can prove the main result of this section.
\begin{Th}\label{lmthbis}
Let $\At$ be a smooth and proper  DG algebra over $K(S)$ and let $d$ be a non-negative integer such that  
 \begin{equation}\label{dimboundp}
HH_m(\At, \At)=0, \;  \text{for every m with}  \; |m| > d.
\end{equation}
 Then the Gauss-Manin connection on the relative periodic cyclic homology $HP_* (\At)$ has regular  singularities  
 and  quasi-unipotent local monodromy of exponent $\leq d+1$. 
 \end{Th}
 \begin{proof}

Using Theorem 1 from  \cite{toen}, there exists a finitely generated $\bZ$-algebra $R\subset K$, a smooth affine scheme $C$ of relative dimension $1$ over $R$ with a geometrically connected generic fiber,  
and a smooth proper DG algebra $\Bt$ over $\cO(C)$ together with an open embedding $C \otimes _R K \mono S$ of curves over $K$ and a quasi-isomorphism 
$\At=\Bt \otimes_{\cO(C)}  K(S)$ 
of DG algebras  over $K(S)$. We can choose $\Bt$ to be term-wise flat over $\cO(C)$.
Since the Hochschild homology $\bigoplus _i HH_i(\Bt, \Bt)$ of a smooth proper DG algebra is finitely generated over $\cO(C)$ replacing  $C$ by a dense open subscheme we may assume
that $\bigoplus _i HH_i(\Bt, \Bt)$ and  $HP_*(\Bt, \Bt)$   are free $\cO(C)$-modules of finite rank. It follows that
$$HH_i(\Bt , \Bt )\otimes_{\bZ} \bF_p \iso HH_i(\Bt \otimes _\bZ \bF_p, \Bt \otimes _{\bZ} \bF_p),$$
$$HH_i(\Bt, \Bt)\otimes_{\cO(C)} K(S) \iso HH_i(\At, \At). $$
Using the Hodge-to-de Rham spectral sequence  it follows that
the periodic cyclic homology also commutes with the base change.
Then by Cor.  \ref{nilpotencygeneral}  $(M,\nabla)=(HP_{*}(\Bt),\nabla_{GM})$
 satisfies the assumptions of the theorem of Katz with $\nu = d+1$ and we are done.
 \end{proof}

\end{document}